\documentclass[a4paper]{amsart}

\usepackage[utf8]{inputenc}
\usepackage{amsmath}
\usepackage{amsfonts}
\usepackage{xypic}
\usepackage{amsthm}
\usepackage{mathrsfs}
\usepackage{amssymb}
\usepackage[all]{xy}
\usepackage{color}
\usepackage{hyperref}
\usepackage{enumitem}

\numberwithin{equation}{section}

\newtheorem{theorem}{Theorem}[section]
\newtheorem{lemma}[theorem]{Lemma}
\newtheorem{proposition}[theorem]{Proposition}
\newtheorem{corollary}[theorem]{Corollary}
\newtheorem{definition}[theorem]{Definition}

\theoremstyle{remark}
\newtheorem{remark}[theorem]{Remark}

\newcommand{\st}{\mathrm{st}}
\newcommand{\sst}{\mathrm{sst}}
\newcommand{\Spec}{\mathrm{Spec}}
\newcommand{\Mod}{\mathbf{Mod}}

\newcommand{\Hilb}{\mathrm{Hilb}}
\newcommand{\Quot}{\mathrm{Quot}}
\newcommand{\Grass}{\mathrm{Grass}}
\newcommand{\supp}{\mathrm{supp}}
\newcommand{\End}{\mathrm{End}}

\newcommand{\Ext}{\mathrm{Ext}}
\newcommand{\Hom}{\mathrm{Hom}}
\newcommand{\Gr}{\mathrm{Gr}}
\newcommand{\Sym}{\mathrm{Sym}}
\newcommand{\sym}{\mathfrak{S}}
\newcommand{\Groupoids}{(\mathbf{Grpds})}
\newcommand{\Sets}{(\mathbf{Sets})}
\newcommand{\Aff}{(\mathbf{Aff})}    
\newcommand{\op}{\mathrm{op}}

\newcommand{\gr}{\mathrm{gr}}
\newcommand{\id}{\mathbf{1}}
\newcommand{\rk}{\mathrm{rk}}
\newcommand{\Dol}{\mathrm{Dol}}
\newcommand{\Hod}{\mathrm{Hod}}

\newcommand{\DR}{\mathrm{DR}}
\newcommand{\Betti}{\mathrm{B}}
\newcommand{\GL}{\mathrm{GL}}
\newcommand{\SL}{\mathrm{SL}}
\newcommand{\Slash}{/\!\!/}

\newcommand{\LF}{\mathrm{LF}}

\newcommand{\im}{\mathrm{im} \,}
\newcommand{\coker}{\mathrm{coker}}

\newcommand{\Hhom}{\mathscr{H}om}

\newcommand{\Aut}{\mathrm{Aut}}
\newcommand{\Stab}{\mathrm{Stab}}
\newcommand{\coh}{\mathrm{coh}}

\newcommand{\Dd}{\mathcal{D}}
\newcommand{\Ee}{\mathcal{E}}
\newcommand{\Ff}{\mathcal{F}}
\newcommand{\Gg}{\mathcal{G}}

\newcommand{\Ll}{\mathcal{L}}

\newcommand{\Oo}{\mathcal{O}}
\newcommand{\Pp}{\mathcal{P}}

\newcommand{\Ss}{\mathcal{S}}
\newcommand{\Tt}{\mathcal{T}}
\newcommand{\Uu}{\mathcal{U}}
\newcommand{\Vv}{\mathcal{V}}
\newcommand{\Ww}{\mathcal{W}}

\newcommand{\FF}{\mathbb{F}}
\newcommand{\LL}{\mathbb{L}}
\newcommand{\HH}{\mathbb{H}}

\newcommand{\KK}{\mathbb{K}}
\newcommand{\PP}{\mathbb{P}}
\newcommand{\TT}{\mathbb{T}}
\newcommand{\VV}{\mathbb{V}}

\newcommand{\M}{\mathbf{M}}
\newcommand{\N}{\mathbf{N}}
\renewcommand{\L}{\mathbf{L}}

\newcommand{\R}{\mathbf{R}}

\newcommand{\Lll}{\mathscr{L}}
\newcommand{\Hhh}{\mathscr{H}}
\newcommand{\Mmm}{\mathscr{M}}
\newcommand{\Nnn}{\mathscr{N}}

\newcommand{\Ttt}{\mathscr{T}}

\renewcommand{\AA}{\mathbb{A}}

\newcommand{\CC}{\mathbb{C}}

\newcommand{\ZZ}{\mathbb{Z}}

\newcommand{\quotient}[2]{{\raisebox{.2em}{\thinspace $#1$}\left / \raisebox{-.15em}{ $#2$}\right.}}

\newcommand\Quotient[2]{
        \mathchoice
            {
                \text{\raise1ex\hbox{\thinspace $#1$}\Big{/} \lower1ex\hbox{$#2$} \thinspace}%
            }
            {
                #1\,/\,#2
            }
            {
                #1\,/\,#2
            }
            {
                #1\,/\,#2
            }
    }

\newcommand\GIT[2]{
        \mathchoice
            {
                \text{\raise1ex\hbox{\thinspace $#1$}\Big{/}\!\!\!\!\Big{/} \lower1ex\hbox{$#2$} \thinspace}%
            }
            {
                #1\,/\,#2
            }
            {
                #1\,/\,#2
            }
            {
                #1\,/\,#2
     a       }
    }

\newcommand{\onproduct}[1]{\mathscr{S}(#1)}

\newcommand{\morph}[6]{\begin{array}{cccc} #6: & #1  & \stackrel{#5}{\longrightarrow} &  #2  \\ & #3 & \longmapsto & #4  \end{array}}

\title{\bf Moduli spaces of $\Lambda$-modules on abelian varieties}

\author{Emilio Franco}
\address{Emilio Franco \\ CMUP (Centro de Matem\'atica da Universidade do Porto) \\ Universidade do Porto \\ Rua do Campo Alegre 1021/1055 \\ 4169-007, Porto (Portugal)}
\email{emilio.franco@fc.up.pt}

\author{Pietro Tortella}
\address{Pietro Tortella \\ SISSA (Scuola Internazionale Superiore di Studi Avanzati) \\
Ma\-the\-ma\-tical Physics Sector \\
Via Bonomea, 265 \\ 34136 Trieste (Italy)}
\email{pietro.tortella@gmail.com}

\date{\today}

\begin{document}


\thanks{First author is supported by FAPESP (Funda\c{c}\~{a}o de Amparo \`{a} Pesquisa do Estado de S\~{a}o Paulo) through grants 2012/16356-6 and 2015/06696-2. Second author was supported by FAPESP grant 2011/17593-9, the PRIN "Geometria delle variet\`a algebriche", and the INdAM group GNSAGA. The first author wants to thank Imperial College London for the hospitality during the time while this article was finished.}

\maketitle

\begin{abstract}
We study the moduli space $\M_X(\Lambda, n)$ of semistable $\Lambda$-modules of vanishing Chern classes over an abelian variety $X$, where $\Lambda$ belongs to a certain subclass of $D$-algebras. In particular, for $\Lambda = \Dd_X$ (resp. $\Lambda = \Sym^\bullet \Tt X$) we obtain a description of the moduli spaces of flat connections (resp. Higgs bundles).

We give a description of $\M_X(\Lambda, n)$ in terms of a symmetric product of a certain fibre bundle over the dual abelian variety $\hat{X}$. We also give a moduli interpretation to the associated Hilbert scheme as the classifying space of $\Lambda$-modules with extra structure. Finally, we study the non-abelian Hodge theory associated to these new moduli spaces.
\end{abstract}

\tableofcontents

\section{Introduction}
\label{sc introduction}

Beilinson and Bernstein \cite{beilinson&bernstein} introduced the notion of {\em $D$-algebra} over a smooth projective variety $Y$. A $D$-algebra is a sheaf of $\Oo_Y$-algebras over $Y$ with pro\-per\-ties analogous to those of $\Dd_Y$, the sheaf of differential operators over $Y$.

Given a $D$-algebra $\Lambda$, one can study the category $\Mod(\Lambda)$ of ($\Oo_Y$-qua\-si\-cohe\-rent) $\Lambda$-modules. Many important geometrical objects over $Y$ can be understood as $\Lambda$-modules for a certain $D$-algebra. In particular, vector bundles with a flat connec\-tion and Higgs bundles are objects in $\Mod(\Lambda^\DR)$ and $\Mod(\Lambda^\Dol)$ respectively, for $\Lambda^{\DR} = \Dd_Y$ the algebra of differential operators and $\Lambda^{\Dol} = \Sym^\bullet(\Tt Y)$, being $\Tt Y$ the tangent sheaf of $Y$.

In \cite{simpson1}, Simpson studied the moduli theory for the category $\Mod(\Lambda)$, cons\-tructing the moduli space $\M_Y(\Lambda, n)$ of rank $n$ semistable $\Lambda$-modules over any smooth projective variety $Y$.

Explicit description of the moduli spaces of $\Lambda$-modules are very rare in the li\-te\-ra\-tu\-re. When the base is an elliptic curve $X$, one has some important results: in \cite{franco&garciaprada&newstead} (see the previous work \cite{thaddeus} for the description of their normalization), the Dolbeault and the De Rham moduli spaces are described as $\Sym^n(\Tt^*\hat{X})$ and $\Sym^*(X^\natural)$, where $\hat{X}$ denotes the dual abelian variety and $X^\natural$ denotes the moduli of line bundles on $X$ equipped with a flat connection. Gorsky, Nekrasov and Rubtsov \cite{gorsky&nekrasov&rubtsov} studied a rigidification of the Dolbeault mo\-du\-li problem introducing the notion of what we call {\em marked Higgs bundle}; in the case of an elliptic curve these objects are equivalent to a Higgs bundle with some extra parabolic structure. They described their mo\-du\-li space $\N_X^\Dol(n) \cong \Hilb^n(\Tt^*\hat{X})$. Confirming a conjecture of Boalch \cite{boalch}, Groechenig \cite{groechenig}  constructed five fa\-mi\-lies of parabolic Higgs bundles over elliptic curves, including a new approach to the mo\-du\-li space $\N_X^\Dol(n)$. The description in \cite{groechenig} makes use of the derived equivalence of categories 
\begin{equation} \label{eq intro Phi^Dol}
\Dd^b_\coh(\Mod(\Lambda^\Dol)) \cong \Dd^b_\coh(\Tt^*\hat{X})
\end{equation}
obtained as the composition of the relative Fourier-Mukai transform, $\Dd^b_\coh(\Tt^*\hat{X}) \cong \Dd^b_\coh(\Tt^*X)$, with $\Dd^b_\coh(\Tt^*X) \cong \Dd^b_\coh(\Mod(\Lambda^\Dol))$, the Beauville-Narasimhan-Ra\-ma\-nan correspondence. Groechenig provides, as well, a description of the associated mo\-du\-li spaces of parabolic local systems (which in our notation corres\-ponds to the moduli space of {\em marked vector bundles with flat connections}), as $\N_X^\DR (n) \cong \Hilb^n(X^\natural)$. For this result, instead of \eqref{eq intro Phi^Dol}, one employs the derived equivalence 
\begin{equation} \label{eq intro laumon rothstein}
\Dd^b_\coh(\Mod(\Lambda^\DR)) \cong \Dd^b_\coh(X^\natural),
\end{equation}
which was proved independently by Laumon \cite{laumon} and Rothstein \cite{rothstein}. Note that \eqref{eq intro laumon rothstein} constitutes the proof of the Geometric Langland Correspondence in the abelian case.

Polishchuk and Rothstein \cite{polishchuk&rothstein} generalized \eqref{eq intro laumon rothstein} with the construction of an analog of the Fourier-Mukai transform for $\Lambda$-modules over an abelian variety $X$. They obtain a derived equivalence    
\begin{equation} \label{eq intro polishchuk-rothstein}
\Dd^b_\coh(\Mod(\Lambda)) \cong  \Dd^b_\coh(\Mod(\hat{\Lambda})),
\end{equation}
where $\hat{\Lambda}$, the {\em Fourier-Mukai dual} of $\Lambda$, is a $D$-algebra (unique up to isomorphism) over the dual abelian variety $\hat{X}$.

We should emphasize that, when the abelian variety $X$ has dimension greater than $1$, the extra structure considered in a marked Higgs bundle can not be understood as a parabolic structure. This circumstance justifies the introduction of our notation.

\

In this article we describe the moduli space $\M_X(\Lambda, n)$ of semistable $\Lambda$-modules with vanishing Chern classes over an abelian variety $X$, of arbitrary dimension, when $\Lambda$ belongs to a certain subclass of $D$-algebras. This generalizes the works cited so far in various directions. First of all, we will work on an abelian variety of any dimension; then we work for any $\Lambda$ satisfying some condition; finally, we give the definition of {\em marked $\Lambda$-module}, which generalizes the correspondence of \cite{gorsky&nekrasov&rubtsov} and \cite{groechenig} between the rigidified moduli spaces and the Hilbert schemes of points to a much wider class of cases. 

The class of $D$-algebras we work with is the following: fix a vector space $V$; any $\alpha \in V^* \otimes H^0(X,\Tt X)$ defines a Lie algebroid structure $\VV_\alpha$ on $\Oo_X \otimes V$; our $D$-algebra is the universal enveloping algebra $\Lambda^\alpha = \Uu(\VV_\alpha)$. This choice is motivated by the fact that, since on an abelian variety the tangent bundle is trivial, the $D$-algebras of main interest satisfy these assumptions. Indeed, Higgs bundles, flat connections, $\tau$-connections, integrable connections along foliations given by trivial subbundles of $\Tt X$, Poisson modules and co-Higgs bundles, are important examples of $\Lambda^\alpha$-modules for particular choices of $V$ and $\alpha$. Moreover, for these $D$-algebras $\Lambda^\alpha$, one has an explicit description of the Fourier-Mukai dual $\hat{\Lambda}_\alpha$, that will be the main ingredient to describe the moduli spaces.
Indeed, associated to $\alpha$ one defines an extension of group schemes
\[
0 \to V^* \to \hat{X}^\alpha \stackrel{\chi}{\to} \hat{X} \to 0
\]
and it turns out that $\hat{\Lambda}_\alpha = \chi_* \Oo_{\hat{X}^\alpha}$.
Using this, \eqref{eq intro polishchuk-rothstein} reads
\begin{equation} \label{eq intro Phi^alpha}
\Dd^b_\coh(\Mod(\Lambda^\alpha)) \cong \Dd^b_\coh(\hat{X}^\alpha),
\end{equation}
under which topologically trivial rank $1$ $\Lambda^\alpha$-modules correspond to geometric points of $\hat{X}^\alpha$. We can recover $\Lambda^\DR$ and $\Lambda^\Dol$ as the universal enveloping algebras of the Lie algebroids obtained by setting $V = H^0(X,\Tt X)$ and $\alpha$ to be, respectively, the identity map and the $0$ endomorphism. Note that, in these cases, the previous construction provides $X^\DR = X^\natural$ and $X^\Dol = \Tt^*\hat{X}$. In the first case, when $\alpha = \id$, \eqref{eq intro Phi^alpha} corresponds to the Laumon-Rothstein transform. We then observe that \eqref{eq intro Phi^alpha} in the case of $\alpha = 0$, coincides with the classical limit of the Geometric Langlands Correspondence.

We then study the stability of $\Lambda^\alpha$-modules over abelian varieties, proving that every semistable $\Lambda^\alpha$-module with vanishing Chern classes over an abelian variety arises as the extension of topologically trivial rank $1$ $\Lambda^\alpha$-modules. As a consequence of that and \eqref{eq intro Phi^alpha}, we identify the moduli stack of rank $n$ semistable $\Lambda^\alpha$-modules $\Mmm^{\sst}_X(\Lambda^\alpha, n)$ with the stack of torsion sheaves on $\hat{X}^\alpha$ with length $n$. This implies the following:

\begin{theorem} \label{tm description of M_X}
Let $X$ be an abelian variety. One has the isomorphism of quasi-projective varieties 
\[
\M_X(\Lambda^{\alpha}, n) \cong \Sym^n(\hat{X}^\alpha).
\]
\end{theorem}

Then, we introduce the notion of {\em marked $\Lambda^\alpha$-modules}; these are triples $(\Ff, \theta, \sigma)$, with $(\Ff,\theta)$ a $\Lambda^\alpha$-module and $\sigma$ a point in the fibre of $\Ff$ over $x_0$ the identity of $X$. We will define a notion of stability for marked $\Lambda$-modules,  showing the existence of the associated moduli space $\N_X(\Lambda^{\alpha}, n)$, and, by studying the Polishchuk-Rothstein transform for marked $\Lambda^\alpha$-modules, we will obtain:

\begin{theorem} \label{tm N = Hilb}
Let $X$ be an abelian variety. Then,
\[
\N_X(\Lambda^{\alpha}, n) \cong \Hilb^n(\hat{X}^\alpha)
\]
is an isomorphism of quasi-projective varieties.
\end{theorem}

The Hilbert-Chow morphism $\Hilb^n(\hat{X}^\alpha) \to \Sym^n(\hat{X}^\alpha)$ will correspond to for\-getting the marking, $(\Ff,\theta,\sigma) \mapsto (\Ff,\theta)$.

Finally, we focus on the study of non-abelian Hodge theory for marked objects. We define the moduli space $\N_X^\Betti(n) := (\Hom(\pi_1(X), \GL(n,\CC)) \times \CC^n) /\!\!/ \GL(n,\CC)$ of marked representations. Recalling that the fundamental group of an abelian variety is $\pi_1(X) \cong \ZZ^{2d}$, the previous space is isomorphic to $\Hilb^n((\CC^*)^{2d})$ as a consequence of work of Henni and Jardim \cite{henni&jardim}. We obtain a complex analytic isomorphism between $\N^\Betti_X(n)$ and $\N_X^\DR(n)$, and we show that $\N_X^\DR(n)$ and $\N_X^\Dol(n)$ are deformation equivalent.

\

The paper is structured as follows: in Section \ref{sc review} we recall, on the one hand, results on $D$-algebras and moduli spaces of $\Lambda$-modules, and on the other hand, results on the Fourier-Mukai transform for $\Lambda$-modules following the work of Laumon, Rothstein and Polishchuk \cite{laumon, rothstein, polishchuk&rothstein}. In Section \ref{sc moduli spaces} we study the moduli spaces of $\Lambda$-modules on abelian varieties: first we analyze the rank $1$ case, then we show that higher rank semistable objects are extension of the rank $1$; in turn, this allows to describe the moduli spaces as the symmetric product of the rank $1$ case; we then introduce marked objects in Section \ref{sc marked} to provide a moduli of the associated Hilbert schemes. In Section \ref{sc non-abelian Hodge theory} we give our main application towards non-abelian Hodge theory for abelian varieties.

\

\noindent
{\em Acknowledgements.} We thank Michael Groechenig for his help and valuable suggestions. Our gratitude also goes to Ugo Bruzzo, Marcos Jardim, Anton Mellit and Jacopo Scalise for discussions and comments. We acknowledge the referees for pointing out mistakes on previous versions of the paper, improving it with their remarks and recommendations.

\section{Lie algebroids and $\Lambda$-modules}
\label{sc review}

\subsection{$D$-algebras and Lie algebroids}
\label{sc D-algebras tau-connections}

Let $Y$ be a smooth variety over $\CC$. Let us recall some de\-fi\-ni\-tions from \cite{beilinson&bernstein}. A {\em differential $\Oo_{Y}$-bimodule} is a quasicoherent sheaf on $Y \times Y$ supported on the diagonal $\Delta(Y) \subset Y \times Y$. We can regard a differential $\Oo_{Y}$-bimodule as a sheaf of $\Oo_Y$-bimodules over $Y$. A {\em $D$-algebra} on $Y$ is a sheaf of associative algebras $\Lambda$ on $Y$ equipped with a morphism of algebras $i : \Oo_{Y} \to \Lambda$ such that the product of $\Lambda$ gives it a differential $\Oo_{Y}$-bimodule structure. This implies that $\Lambda$ comes with an increasing filtration 
\begin{equation} \label{eq filtration of Lambda}
0 = \Lambda_{-1} \subset \Lambda_0 \subset \Lambda_1 \subset \Lambda_2 \subset \dots
\end{equation}
such that $\Lambda = \bigcup_i \Lambda_i$ and for any $f$ in $\Oo_Y$ and $\lambda \in \Lambda_i$ one has $f \cdot \lambda - \lambda \cdot f \in \Lambda_{i-1}$. We denote $\Gr_i \Lambda = \Lambda_i/\Lambda_{i-1}$ and $\Gr_\bullet \Lambda := \bigoplus \Gr_i \Lambda_i$. Recalling that $\Lambda$ is a differential $\Oo_Y$-bimodule, we denote by $\onproduct{\Lambda}$ the associated quasicoherent sheaf on $Y\times Y$ supported on the diagonal.

We will focus on $D$-algebras that are \emph{almost polynomial} (cf. \cite[Section 2]{simpson1}), namely those $D$-algebras $\Lambda$ such that $\Gr_1 \Lambda$ is a locally free $\Oo_{Y}$-module and whose associated graded algebra is isomorphic to the symmetric product over the first graded piece, $\Gr_\bullet \Lambda = \Sym_{\Oo_{Y}}^\bullet (\Gr_1 \Lambda)$. 

Almost polynomial $D$-algebras may be described in terms of Lie algebroids. Recall that a {\em Lie algebroid} on $Y$ is a triple $\LL = (\Ll, a, [\cdot,\cdot])$, where $\Ll$ is a locally free $\Oo_{Y}$-module equipped with a Lie bracket $[\cdot,\cdot] : \Ll \otimes \Ll \to \Ll$, and a morphism to the tangent sheaf, $a : \Ll \to \Tt {Y}$, called {\it anchor}, satisfying
\[
[\ell_1 , f \ell_2 ] = f \cdot [\ell_1 , \ell_2] + a(\ell_1)(f)\ell_2,
\]
for every $\ell_1 , \ell_2 \in \Ll$ and $f \in \Oo_Y$. Observe that, given a Lie algebroid $\LL$, its universal enveloping algebra $\Uu(\LL)$ is an almost polynomial $D$-algebra.

Let $\Omega^k_\LL = \bigwedge^k \Ll^*$ denote the sheaf of $k$-$\LL$-forms. One can define the Lie algebroid differential $d_\LL : \Omega^k_\LL \to \Omega^{k+1}_\LL$ setting
\begin{align*}
(d_\LL \theta) (u_1, \dots, u_{k+1}) := & \sum_{i = 1}^{k + 1} (-1)^{i + 1} a(u_i) \theta (u_1, \dots, \hat{u}_i, \dots u_{k+1}) +
\\
& + \sum_{i < j} (-1)^{i + j} \theta \left ( \{ u_i, u_j \}, u_1, \dots, \hat{u}_i, \hat{u}_j, \dots, u_{k+1} \right ),
\end{align*}
for each $\theta \in \Omega^k_\LL$ and any $u_1, \dots, u_{k+1}$ (local) sections of $\Ll$. The Lie algebroid differential squares to zero, $d_\LL = 0$, and this defines the complex $\Omega^\bullet_\LL$. We denote by $\tau^{\geq r} \Omega^\bullet_\LL$ the {\it b\^ete filtration} of the complex $\Omega_\LL^\bullet$, which is the complex 
\[
(\tau^{\geq r} \Omega^\bullet_\LL)^k = \left\{  \begin{array}{ll}
     0 &  \text{  if $k < r$}  \\
     \Omega_\LL^k &  \text{  if $k\geq r$ }.
\end{array}
\right.
\]

Let $\mathfrak{U} = \{ U_i \}$ be a sufficiently fine open covering of $Y$, such that we have an isomorphism between the sheaf and \v{C}ech cohomology over it. Consider the double complex $K^{p,q}_{\LL} := \check{C}^q(\mathfrak{U}, \Omega^p_\LL)$, with differentials given by $d_\LL$ and the \v{C}ech coboundary, and recall that its associated total complex $T^\bullet_\LL$ computes the hypercohomology of $\Omega^\bullet_\LL$. 
Remark that the hypercohomology of the complex $\tau^{\geq r} \Omega^\bullet_\LL$ is isomorphic to the cohomology of the complex of vector spaces 
\begin{equation} \label{eq cohomology of truncated complex}
T_{\tau^{\geq r}\Omega_\LL^\bullet}^k=  \bigoplus_{p + q = k, p \geq r} K^{p,q}_{\LL}.
\end{equation}

The relation between Lie algebroids and $D$-algebras is stated in the following lemma: 

\begin{lemma}[cf. \cite{beilinson&bernstein}, \cite{tortella} Theorem 2] \label{lemma_Lambda-algebroids}
Let $\Ll$ be a locally free $\Oo_Y$-module of finite rank. There is a bijective correspondence between isomorphism classes of:
\begin{enumerate}[label=(\roman*)]
\item \label{it lambda Xi} pairs $(\Lambda,\Xi)$, with $\Lambda$ an almost polynomial $D$-algebra and $\Xi$ an isomorphism of the associated graded algebra $\Gr_\bullet \Lambda$ with the symmetric algebra $\Sym^\bullet_{\Oo_Y}(\Ll)$.
\item \label{it L wtL} pairs $(\LL, \widetilde{\LL})$, with $\LL$ a Lie algebroid structure on $\Ll$ and $\widetilde{\LL}$ a central extension of $\LL$ by $\Oo_Y$; 
\item \label{it L Sigma} pairs $(\LL, \Sigma)$, with $\LL$ a Lie algebroid structure on $\Ll$ and $\Sigma \in \HH^2(Y, \tau^{\geq 1} \Omega^\bullet_\LL)$.
\end{enumerate}
\end{lemma}

The equivalence between \ref{it L wtL} and \ref{it L Sigma} follows from the classification of Lie algebroid extensions. The equivalence between \ref{it lambda Xi} and \ref{it L wtL} goes as follows: given a pair $(\Lambda, \Xi)$ as in \ref{it lambda Xi}, the commutator of elements in $\Lambda$ defines a Lie algebroid structure on $\Lambda_1$ and on $\gr_1\Lambda$. The isomorphism $\Xi$ then defines the Lie algebroid extension
\begin{equation} \label{equation_Lambda}
0 \to  \Oo_Y \to \Lambda_1 \to  \Ll \to 0
\end{equation}
Conversely, to a Lie algebroid extension as in \ref{it L wtL}, 
\[
0 \longrightarrow \Oo_Y \stackrel{c}{\longrightarrow} \widetilde{\LL} \longrightarrow \LL \longrightarrow 0,
\]
one can associate an almost polynomial $D$-algebra $\Uu^0(\widetilde{\LL})$, the {\em reduced enveloping algebra} of $\widetilde{\LL}$, which is the universal enveloping algebra $\Uu(\widetilde{\LL})$ modulo the ideal ge\-ne\-ra\-ted by the identification of the embeddings $\Oo_Y \hookrightarrow \Uu(\widetilde{\LL})$ and $c(\Oo_Y) \hookrightarrow \widetilde{\LL} \hookrightarrow \Uu(\widetilde{\LL})$. For the trivial central extension $\widetilde{\LL} = \LL \oplus \Oo_Y$, the reduced enveloping algebra of $\widetilde{\LL}$ coincides with the universal enveloping algebra of $\LL$, i.e. $\Uu^0(\widetilde{\LL}) = \Uu(\LL)$. In case, we say that we obtain an {\it untwisted} $D$-algebra.

Given $\Lambda$ a $D$-algebra, the sequence \eqref{equation_Lambda} is equivalent to the short exact sequence
\[
0 \to \Delta_* \Oo_Y \to \onproduct{\Lambda_1} \to \Delta_* \Ll \to 0
\]
on $Y\times Y$, where $\Delta : Y \to Y \times Y$ is the diagonal map. One has:

\begin{lemma}[\cite{polishchuk&rothstein} Lemma 5.2] \label{lm description of Ext^1_XxX}
For $\Ll$ a locally free $\Oo_Y$-module of finite rank, one has a canonical isomorphism
\[
\Ext^1_{\Oo_{Y \times Y}}(\Delta_*\Ll, \Delta_*\Oo_Y) \cong \Hom_{\Oo_Y}(\Ll, \Tt Y) \oplus \Ext^1_{\Oo_Y}(\Ll, \Oo_Y).
\]
\end{lemma} 

If we denote by $(a,b)$ the class in $\Hom_{\Oo_Y}(\Ll, \Tt Y) \oplus \Ext^1_{\Oo_Y}(\Ll, \Oo_Y)$ associated to $\onproduct{\Lambda_1}$, then $a$ coincides with the anchor of the Lie algebroid structure on $\Ll$, while $b$ describes $\Lambda_1$ as an extension of $\Oo_Y$-modules.

An important class of algebras are constructed from Lie algebroids supported on the tangent bundle. In the untwisted case, we have of course the {\it algebra of differential operators}, or {\it De Rham $D$-algebra} in Simpson's notation \cite{simpson1, simpson2} $\Dd_Y = \Lambda^\DR$ which arises as the universal enveloping algebra of the canonical Lie algebroid $\left ( \Tt Y , \id, [\cdot, \cdot]\right )$, i.e. the Lie algebroid supported on $\Tt Y$ obtained after setting the anchor to be the identity morphism. The abelianization of $\Lambda^\DR$ is the {\it Dolbeault $D$-algebra}, $\Lambda^\Dol = \Sym^\bullet(\Tt Y)$, which can be obtained as the universal enveloping algebra of the trivial Lie algebroid supported on the tangent bundle, $\left ( \Tt Y, 0, [\cdot, \cdot]  \right )$, where the anchor is the $0$ map. We can construct also a family of $D$-algebras which is a deformation from $\Lambda^\DR$ to $\Lambda^\Dol$. Set, for each $\tau \in \CC$, the Lie algebroid $\left ( \Tt Y , \tau \id, [\cdot, \cdot]\right )$, where the anchor consists on scaling by $\tau$, and define $\Lambda^\tau$ to be the universal enveloping algebra of it. In general, an {\it algebra of twisted differen\-tial operators} (or tdo) is a $D$-algebra defined as $\Lambda^{tdo} = \Uu^0(\widetilde{\TT})$, where $\widetilde{\TT}$ is a central extension of a Lie algebroid $(\Tt Y, a, [ \cdot , \cdot])$ supported on the tangent bundle. 

More examples of $D$-algebras can be constructed starting from a smooth holomorphic foliation $i: F \hookrightarrow \Tt Y$. One can naturally define the Lie algebroid $\FF = (F , i, [\cdot, \cdot])$ using the inclusion of tangent bundles as the anchor, and set $\Lambda^F$ to be the universal enveloping algebra of $\FF$.

Another important class of $D$-algebras come from Poisson geometry. Given a translation invariant Poisson bivector $\Pi \in \bigwedge^2 H^0(Y, \Tt Y)$, the contraction with $\Pi$ defines a Lie algebroid structure on $\Tt^*_Y$ and we set $\Lambda^\Pi$ to be its universal enveloping algebra. We denote by $\Lambda^{co-Higgs}$ the $D$-algebra associated to $\Pi = 0$.

\subsection{Moduli spaces of $\Lambda$-modules}
\label{sc Simpson construction}

Let $\Lambda$ be a $D$-algebra on a smooth projective scheme $Y$, a {\em $\Lambda$-module} $\Ff$ is a sheaf on $Y$ with an action $\Lambda \otimes \Ff \to \Ff$ defining a module structure. Note that the structural morphism $\Oo_{Y} \hookrightarrow \Lambda$ provides a natural $\Oo_{Y}$-module structure on $\Ff$. We denote by $\Mod(\Lambda)$ the category of $\Lambda$-modules on $Y$ which are quasicoherent as $\Oo_{Y}$-modules.

When $\Lambda = \Uu(\LL)$, using the correspondence of Lemma \ref{lemma_Lambda-algebroids}, one can describe $\Lambda$-modules over a scheme $Y$ in terms of flat $\LL$-connections. Given a Lie algebroid $\LL$ and an $\Oo_Y$-module $\Ff$, an {\em $\LL$-connection} on $\Ff$ is a map of $\CC_Y$-modules
$$
\theta: \Ff \to \Ff \otimes \Omega^1_\LL
$$
satisfying the Leibniz rule
\[
\theta(h s) = h \theta(s) + s \otimes d_\LL h,
\]
for every $h \in \Oo_Y$ and $s \in \Ff$.

The {\em curvature} of an $\LL$-connection is the $2$-$\LL$-form $F_\theta \in H^0(Y, \Omega^2_\LL \otimes \End(\Ff))$ defined by 
$$
F_\theta(a_1,a_2) = [\theta_{a_1} , \theta_{a_2}] - \theta_{[a_1,a_2]},
$$
where $a_1,a_2$ are sections of $\LL$, and $\theta_a$ denotes the contraction of $\theta$ with $a$, seen as a differential operator on $\Ff$. A connection is said to be {\em flat} when its curvature vanishes; we will also call a flat $\LL$-connection $(\Ff,\theta)$ an {\em $\LL$-module}.

\begin{lemma}[\cite{simpson1} Lemma 2.13] 
\label{lm Lambda-modules and flat Lambda-connections}
Let $\LL$ be a Lie algebroid on the scheme $Y$, $\Lambda = \Uu(\LL)$ and $\Ff$ a coherent $\Oo_Y$-module.

Then a $\Lambda$-module structure on $\Ff$ is equivalent to an $\LL$-module structure.
\end{lemma}

After Lemma \ref{lm Lambda-modules and flat Lambda-connections}, we will refer interchangeably to $\Uu(\LL)$-modules and $\LL$-modules. Note that a subsheaf $\Ff' \subset \Ff$ is a $\Lambda$-submodule if and only if it is {\em preserved by} $\theta$, that is,
\[
\theta(\Ff') \subseteq \Ff' \otimes_{\Oo_Y} \Ll^*.
\]

Important examples of $\Lambda$-modules arise when we consider $\Lambda^\DR = \Dd_Y$ and $\Lambda^\Dol$. In the first case, $\Oo_Y$-coherent $\Lambda^\DR$-modules ($\Oo_Y$-coherent $D$-modules) are equivalent to vector bundles with flat connections on $Y$, while in the second, $\Lambda^\Dol$-modules are equivalent to Higgs sheaves over $Y$. For each $\tau \in \CC$, $\Lambda^\tau$-modules are equivalent to flat $\tau$-connections (see \cite{simpson1, simpson2} for a definition), from whom Higgs sheaves and flat connections are particular cases when $\tau = 0$ and $\tau = 1$. Another important class of examples are twisted $D$-modules, which are modules for an algebra of twisted differen\-tial operators $\Lambda^{tdo}$, and vector bundles with flat connections along a foliation $F \subset \Tt Y$, which can be identified with $\Oo_Y$-coherent $\Lambda^F$-modules. For a translation invariant Poisson bivector $\Pi \in \bigwedge^2 H^0(Y, \Tt Y)$, $\Lambda^\Pi$-modules are called $\Pi$-Poisson modules and, when the underlying $\Oo_Y$-module is locally free, they corres\-pond to holomorphic bundles for a generalized holomorphic structure over $Y$ constructed from $\Pi$ (see \cite[Section 5.3]{tortella} for instance). Furthermore, when $\Pi = 0$, we have that $\Lambda^{co-Higgs}$-modules are identified with co-Higgs sheaves \cite{rayan1, rayan2}.

Let us denote by $P_\Ff$ the Hilbert polynomial of the sheaf $\Ff$. It is defined by the condition $P_\Ff(n) = \dim H^0(Y,\Ff(n))$ for $n\gg 0$. The coefficient of the leading term is $r/d!$, where $r = \rk(\Ff)$ and $d = \dim(\Ff)$. Write $\mu(\Ff)/ d!$ for the second term.

Suppose now that $Y$ is a projective variety over $\Spec(k)$. Following \cite{simpson1}, a $\Oo_Y$-coherent $\Lambda$-module $(\Ff, \theta)$ is {\em (semi)stable} if it is of pure dimension and for any proper $\Lambda$-submodule $\Ff' \subset \Ff$ there exists an $N \in \ZZ^{>0}$ such that
\[
\frac{P_{\Ff'}(n)}{\rk(\Ff')} (\leq) < \frac{P_\Ff(n)}{\rk(\Ff)},
\] 
for any $n \geq N$.

A $\Lambda$-module $(\Ff, \theta)$ is said {\em $\mu$-(semi)stable} if it is of pure dimension and for any proper $\Lambda$-submodule, $\Ff' \subset \Ff$, one has
\[
\frac{\mu(\Ff')}{\rk(\Ff')} (\leq) < \frac{\mu(\Ff)}{\rk(\Ff)}.
\] 
Note that semistability implies $\mu$-semistability, whereas $\mu$-stability implies stability.

Consider now the trivial $S$-scheme $Y_S = Y \times S$ with projection $\pi: Y \times S \to Y$, and $\Lambda$ a $D$-algebra on $Y$. For each geometric point $s \in S$, set $Y_s := Y \times \{s\}$. Then $\pi^* \Lambda$ is a $D$-algebra on $Y_S$, and any $\pi^*\Lambda$-module $(\Ff, \theta)$ is such that the restriction to each fiber $(\Ff, \theta)|_{Y_s} = (\Ff_s, \theta_s)$ is a $\Lambda$-module. 
A $\Lambda$-module $(\Ff, \theta)$ over $Y_S$ is said {\em semistable}, {\em stable}, {\em $\mu$-semistable} or {\em $\mu$-stable}, if $\Ff$ is flat over $S$ and $(\Ff_s, \theta_s)$ are, respectively, semistable, stable, $\mu$-semistable or $\mu$-stable for any $s \in S$.

Let us consider the moduli stack of rank $n$ semistable $\Lambda$-modules over $Y$, this is the $2$-functor from the opposite category of affine schemes $\Aff^\op$ to the category of groupoids $\Groupoids$, 
\[
\Mmm^{\sst}_{Y}(\Lambda,n) \, : \,  \Aff^{\op} \longrightarrow \Groupoids,
\]
associating to each scheme $S$ the category of $n$ semistable $\Lambda$-modules over $Y_S$ and vanishing Chern classes. Let us consider the functor $\Groupoids \to \Sets$ sending each groupoid the set of its isomorphism classes. The composition of our moduli stack with this functor gives us the associated moduli functor,
\begin{equation} \label{eq stack of Lambda-modules}
\quotient{\Mmm^{\sst}_{Y}(\Lambda,n)}{\! \cong} \, :  \, \Aff^{\op} \longrightarrow \Sets.
%
\end{equation}
 
\begin{theorem}[\cite{simpson1} Theorem 4.7] \label{tm existence of moduli space}
There exists a coarse moduli space $\M_{Y}(\Lambda, n)$ for the classification of semistable $\Lambda$-modules on $Y$ of rank $n$ and vanishing Chern classes (equivalently, $\M_{Y}(\Lambda, n)$ corepresents the moduli functor \eqref{eq stack of Lambda-modules}). The geo\-me\-tric points of $\M_{Y}(\Lambda, n)$ represent $\Ss$-equivalence classes of semistable $\Lambda$-modules.
\end{theorem}

We briefly review the proof of \cite[Theorem 4.7]{simpson1}. Consider an integer $r \gg 0$ big enough to ensure the boundedness of the semistable $\Lambda$-modules \cite[Corollary 3.6]{simpson1} and set $V$ to be the vector space $\CC^{\, P_0(r)}$ where $P_0(r)$ is the evaluation at $r$ of the Hilbert polynomial of those sheaves with vanishing Chern classes and rank $n$. Consider as well the sheaf $\Ww = \Lambda_n \otimes_Y \Oo_Y(-r)$, where $\Lambda_n$ is the coherent sheaf underlying the subalgebra appearing in \eqref{eq filtration of Lambda} and the Quot-scheme $\Quot(V \otimes \Ww, n)$. We recall that the Quot-scheme $\Quot(V \otimes \Ww, n)$ is a fine moduli space for the classification of quotient sheaves $V \otimes \Ww \to \Ff \to 0$ of rank $n$ and vanishing Chern classes and we denote by $\Uu \to Y \times \Quot(V \otimes \Ww, n)$ the associated universal family. Observe that, by the construction of $\Ww$, a quotient sheaf $V \otimes \Ww \to \Ff \to 0$ defines a $\Lambda$-module structure on $\Ff$. Equivalently, the Quot-scheme $\Quot(V \otimes \Ww, n)$ classifies isomorphism classes of triples $(\Ff, \theta, a)$ given by a $\Lambda$-module $(\Ff, \theta)$ of rank $n$ and vanishing Chern classes, and an isomorphism of vector spaces $a : H^0(Y, \Ff(m)) \stackrel{\cong}{\to} V$. Take the open subset $Q \subset \Quot(V \otimes \Ww, n)$ given by those $(\Ff, \theta, a)$ such that the associated $\Lambda$-module $(\Ff, \theta)$ is semistable. Note that $\GL(V)$ acts on $Q \subset \Quot(V \otimes \Ww, n)$ but its centre acts trivially, so it suffices to quotient by $\SL(V)$. Grothendieck proved that, for $m \gg 0$, one can construct an embedding 
\[
\Quot(V \otimes \Ww, n) \hookrightarrow \Grass(H^0(Y, V \otimes \Ww(m)), P_0(m)) \hookrightarrow \PP \left ( W_m \right ),
\]
where 
\[
W_m =  \bigwedge^{P_0(m)} H^0(Y, V \otimes \Ww(m))^\vee.
\]
We denote by $\Lll_m$ the pull-back of the tautological bundle $\Oo_{\PP(W_m)}(1)$ under this embedding. Note that both $\SL(V)$ and $\GL(V)$ act on $\Lll_m$, so this defines a linearization for the action of both groups on $\Quot(V \otimes \Ww, n)$. By \cite[Lemma 4.3]{simpson1}, for $m$ big enough, $Q$ is contained in $\Quot(V \otimes \Ww, n)^{sst}_{\Lll_m}$, the set of semistable points for the action of $\SL(V)$ with respect of the linearization $\Lll_m$. Finally, $\M_X(\Lambda, n)$ is constructed as the GIT quotient $Q \Slash \SL(V)$ with respect to the linearization $\Lll_m$.

Simpson defines a certain rigidification of the moduli problem. Fix a closed point $y_0 \in Y$ and denote by $\xi: S \to Y_S$ the $S$-point constant at $y_0$. A {\em framed $\Lambda$-module} for $\xi$ is $(\Ff, \theta, \varphi)$, where $(\Ff, \theta)$ is a $\Lambda$-module and $\varphi$ is an isomorphism $\varphi : \xi^* \Ff \stackrel{\cong}{\longrightarrow} \Oo_S^{\oplus n}$ called framing at $\xi$. A framed $\Lambda$-module $(\Ff, \theta, \varphi)$ is semistable if the underlying $\Lambda$-module is semistable. We say that a semistable $\Lambda$-module satisfies the condition $\LF(\xi)$ if for every closed point $s$ of $S$, the restriction to $Y_s$ of the associated graded object, $\Gr(\Ff, \theta)_s$, is a locally free $\Oo_{Y_s}$-module. 

\begin{theorem}[\cite{simpson1} Theorem 4.10] \label{tm M=R/GL}
There exists an open subscheme $\M_{Y}^{\LF(\xi)}(\Lambda, n)$ of $\M_{Y}(\Lambda, n)$ corepresenting the functor that associates to any scheme $S$ the set of isomorphism classes of semistable $\Lambda$-modules on $Y \times S$ satisfying the condition $\LF(\xi)$.

The functor that associates to every scheme $S$ the set of isomorphism classes of semistable framed $\Lambda$-modules on $Y \times S$ satisfying the condition $\LF(\xi)$ is re\-pre\-sen\-ted by a scheme $\R_{Y}(\Lambda, n, y_0)$ (equivalently $\R_{Y}(\Lambda, n, y_0)$ is a fine moduli space for the classification of semistable framed $\Lambda$-modules).

There is a natural action of $\GL(n,\CC)$ on $\R_{Y}(\Lambda, n, y_0)$ for which one can construct a linearization $L$ such that every point of $\R_{Y}(\Lambda, n, y_0)$ is semistable. The associated GIT quotient is isomorphic to $\M_{Y}^{\LF(\xi)}(\Lambda, n)$,
\begin{equation} \label{eq description of M as a GIT quotient}
\M_{Y}^{\LF(\xi)}(\Lambda, n) \cong \GIT{\R_{Y}(\Lambda, n, y_0)}{\GL(n,\CC),}
\end{equation}
where the closed orbits correspond to $\Lambda$-modules that are direct sum of stable ones. 
\end{theorem} 

Following Simpson \cite{simpson1, simpson2}, we refer to $\R_{Y}(\Lambda, n, y_0)$ as the {\em representation space} of $\Lambda$-modules on $Y$.

\subsection{Lie algebroids and $D$-algebras on abelian varieties}
\label{sc some lie algebroids}

From now on, let $X$ be an abelian variety, and denote by $\hat{X}$ its dual. 

Set
\begin{equation} \label{eq definition of g}
\mathfrak{g} := H^0(X,\Tt X) \cong H^1(\hat{X},\Oo_{\hat{X}})
\end{equation}
and
\begin{equation} \label{eq definition of g*}
\hat{\mathfrak{g}} \cong H^1(X,\Oo_X) \cong H^0(\hat{X},\Tt {\hat{X}}).
\end{equation}
The tangent bundle of $X$ is trivial, and $\Tt X = \Oo_X \otimes \mathfrak{g}$.

We will consider Lie algebroids over $X$ whose underlying bundle is trivial. Then, let $V$ be a vector space and consider the trivial vector bundle $\Oo_X \otimes V$.

\begin{lemma}
A Lie algebroid structure on $\Oo_X \otimes V$ is uniquely determined by a linear map $\alpha: V \to \mathfrak{g}$ and a Lie algebra structure on $V$ satisfying $[V,V] \subseteq \ker \alpha$.
\end{lemma}

\begin{proof}
The anchor $a: (\Oo_X \otimes V) \to \Tt X$ is a morphism of trivial vector bundles, so it has to satisfy $a = 1_{\Oo_X} \otimes \alpha$ for $\alpha: V \to \mathfrak{g}$ a linear map. Then, the Lie algebroid bracket on $\Oo_X \otimes V$ has to take the form
\[
\left [ f\otimes u, g\otimes v \right ]_{\VV_\alpha} = fg \otimes [u,v]_V + f\alpha(u)(g) \otimes v - g \alpha(v)(f) \otimes u
\]
for some Lie algebra bracket $[\cdot,\cdot]_V$ on $V$. The fact that $a$ should respect the brackets implies that $[V,V]_V \subseteq \ker \alpha$.
\end{proof}

We will be interested in Lie algebroids that can be described entirely by the linear map $\alpha$. This leads to the following:
\begin{definition}
We say that a Lie algebroid over $X$ is {\em UTAI (Underlying Trivial and with Abelian Isotropy)} if the underlying $\Oo_X$-module is a trivial vector bundle and the isotropy subalgebroid (that is  the kernel of the anchor) is abelian.

When $X$ is an abelian variety, a UTAI Lie algebroid is supported on the vector bundle $\Oo_X \otimes V$, where $V$ is a vector space, and it is determined univocally by a linear map $\alpha: V \to \mathfrak{g}$. For a given $\alpha \in V^* \otimes \mathfrak{g}$, we denote by $\VV_\alpha$ the associated UTAI Lie algebroid.
\end{definition}

This definition is motivated by the fact that, on an abelian variety, $\Lambda^\Dol$ and $\Lambda^\DR$ arise as the universal enveloping algebras of UTAI Lie algebroids. Indeed, for every $\tau \in \CC$, we have that $\Lambda^\tau$ is UTAI. Another class of UTAI Lie algebroids are those defined by considering $V$ to be a subspace $F \subset \mathfrak{g}$, giving the foliation $F \otimes \Oo_X \subset \Tt Y$. We recall that $\Lambda^F$-modules are equivalent to vector bundles with connections along our foliation. The Lie algebroid structure on $\Tt^*_X$ defined by a translation invariant Poisson bivector $\Pi \in \bigwedge^2 H^0(X, \Tt X)$ is again UTAI, giving the $D$-algebra $\Lambda^\Pi$. Our definition covers also the case of $\Pi = 0$, whose corresponding $D$-algebra is $\Lambda^{co-Higgs}$.

The aim of this section is to classify $D$-algebras on $X$ whose associated Lie algebroid is UTAI. Thanks to Lemma \ref{lemma_Lambda-algebroids}, this amount to understand the cohomology groups $\mathbb H^2(X, \tau^{\geq 1} \Omega^\bullet_{\VV_\alpha})$. 

\

As a first instance, let us study the case where the map $\alpha$ is injective. This is equivalent to consider $V=F \subseteq \mathfrak{g}$ a subspace, defining the foliation $\Oo_X \otimes F \subseteq \Tt X$. We denote by $\FF$ the associated Lie algebroid. 

\begin{lemma} \label{lm split_cohomology_foliation}
One has the foliated Hodge decomposition
\[
\HH^k (X, \Omega^\bullet_\FF) \cong \bigoplus_{p+q = k} H^q(X, \Omega^p_\FF)\ .
\]
\end{lemma}

\begin{proof}
Consider the sheaf of $\FF$-forms $\Omega_\FF^p = \bigwedge^p \FF^* = \Oo_X \otimes \bigwedge^p F$, with the differential $d_\FF:\Omega^p_\FF \to \Omega^{p+1}_\FF$. Remark that the natural morphism of complexes $ \Omega^\bullet_X \to \Omega^{\bullet}_\FF $ is surjective. The b\^ete filtration on $\Omega_\FF^\bullet$ induces a spectral sequence to compute $\HH^k (X, \Omega_\FF^\bullet)$, where the first term is given by $E_{\FF,1}^{p,q} = H^q(X, \Omega^p_\FF)$. 

For each $p$, the surjective morphism $\Omega^p_X \to \Omega^p_\FF$ splits, since one can take a splitting $ \mathfrak{g} \to F$ of complex vector spaces. This implies that the map at the cohomology level $H^q(X, \Omega_X^p) \to H^q(X, \Omega_\FF^p)$ is surjective and splits.

Now, the morphism of complexes $\Omega_X^\bullet \to \Omega_\FF^\bullet$ preserves the b\^ete filtration, so it induces a morphism between each step of the spectral sequences $E_{X,r}^{p,q} \to E_{\FF,r}^{p,q}$ (where $E^{p,q}_{X,r}$ denotes the Hodge-to-deRham spectral sequence of $X$). In particular, one has the following diagram 
\[
\xymatrix{ 
E_{X,1}^{p,q} \ar[r] \ar[d]_{d_{X,1}} & E_{\FF,1}^{p,q} \ar[d]^{d_{\FF,1}} \ar[r] & 0\\
E_{X,1}^{p+1,q} \ar[r] & E_{\FF,1}^{p+1,q} \ar[r] & 0.
}
\]
By the Hodge theory of $X$, $d_{X, 1} = 0$ while $E^{p,q}_{X, 1} \to E^{p, q}_{\FF, 1}$ coincides with $H^q(X, \Omega^p_X) \to H^q(X, \Omega^p_{\FF})$, which is onto. It follows that $d_{\FF, 1} = 0$ and one has the foliated Hodge decomposition.
\end{proof}

We can now study the hypercohomology of a general UTAI Lie algebroid $\VV_\alpha$.

\begin{proposition} \label{pr split_cohomology}
One has the isomorphisms
\begin{equation} \label{eq hypercohomology of VV_alpha}
\HH^k(X, \Omega^\bullet_{\VV_\alpha}) \cong \bigoplus_{p+q=k} H^q(X, \Omega_{\VV_\alpha}^p)
\end{equation}
and 
\begin{equation} \label{eq truncated hypercohomology of VV_alpha}
\HH^k(X, \tau^{\geq r} \Omega^\bullet_{\VV_\alpha}) \cong \bigoplus_{p + q = k, p \geq r} H^q(X, \Omega_{\VV_\alpha}^p).
\end{equation}
\end{proposition}

\begin{proof}
For $\VV_\alpha$ any UTAI Lie algebroid on $X$, let $K = \ker \alpha$, $F = \operatorname{Im} \, \alpha$, and $\KK$, $\FF$ denote the associated trivial bundles. Fix a splitting $s: F \to V$; this induces splittings of the $\VV_\alpha$-forms 
\[
\Omega_{\VV_\alpha}^k \cong \bigoplus_{p+q = k} \Omega^p_{\FF} \otimes \Omega_{\KK}^q. 
\]
Since $\KK$ is an abelian Lie algebroid, it has a natural action of $\FF$ onto it. This induces an $\FF$-module structure on the $\Omega_\KK^q$, and one      may interpret the sheaves $\Omega_\FF^p \otimes \Omega_\KK^q$ as the sheaves of $\Omega_\KK^q$-valued $p$-$\FF$-forms (and to stress this interpretation, we use the notation $\Omega_\FF^p(\Omega_\KK^q)$ for these sheaves). By the results of \cite{BMRT} follows that, the differential $d_{\VV_\alpha}: \Omega_{\VV_\alpha}^k \to \Omega_{\VV_\alpha}^{k+1}$ coincides with $d_\FF : \Omega_\FF^p(\Omega_\KK^q) \to \Omega_\FF^{p+1}(\Omega_\KK^q)$. Then, the complex $\Omega^\bullet_{\VV_\alpha}$ is the direct sum of the complexes $\Omega_\FF^{\bullet - q}(\Omega_\KK)^q$, and one obtains the isomorphism of vector spaces
%
%
\[
\HH^k(X, \Omega^\bullet_{\VV_\alpha}) \cong \bigoplus_{q=0}^k \HH^{k-q}(X, \Omega^\bullet_{\FF}(\Omega^q_{\KK})) \cong \bigoplus_{q=0}^k \HH^{k-q}(X, \Omega^\bullet_{\FF}) \otimes \bigwedge^q K^* \ ,
\]
where the last equality follows from the fact that $\Omega^q_{\KK}$ is the trivial vector bundle. 

On the other hand, since $\KK$ is a trivial vector bundle, one has the isomorphisms $H^l(X, \Omega^p_\FF \otimes \Omega^q_\KK) \cong H^l(X, \Omega^p_\FF) \otimes \bigwedge^q K^*$. Then, \eqref{eq hypercohomology of VV_alpha} follows from Lemma \ref{lm split_cohomology_foliation}, and \eqref{eq truncated hypercohomology of VV_alpha} is a direct consequnce of this and \eqref{eq cohomology of truncated complex}. 
\end{proof}

This allow us to classify UTAI $D$-algebras. 

\begin{proposition} \label{pr classification of UTAI D-algebras}
There is a bijective correspondence between $\GL(V)$-orbits of triples of the form
\[
(\alpha, \beta , \gamma ) \in (V^* \otimes \mathfrak{g}) \oplus (V^* \otimes \hat{\mathfrak{g}}) \oplus \bigwedge^2 V^*\ ,
\]
and isomorphism classes of almost polynomial $D$-algebras whose associated Lie algebroid is UTAI with underlying vector bundle $V \otimes \Oo_X$.

Hence, a $D$-algebra of this form is abelian if and only if it is associated to a triple with $\alpha = 0$ and $\gamma = 0$.
\end{proposition}

\begin{proof}
Thanks to the equivalence between \ref{it L wtL} and \ref{it L Sigma} of Lemma \ref{lemma_Lambda-algebroids}, we know that central extensions of Lie algebroids of the form
\begin{equation} \label{eq central extension of VV_alpha}
0 \longrightarrow \Oo_X \stackrel{c}{\longrightarrow} \widetilde{\VV}_\alpha \longrightarrow \VV_\alpha \longrightarrow 0,
\end{equation}
are classified by $\HH^2(X, \tau^{\geq 1} \Omega_{\VV_\alpha}^\bullet)$. After \eqref{eq truncated hypercohomology of VV_alpha}, this space decomposes as
\[
\HH^2(X, \tau^{\geq 1} \Omega_{\VV_\alpha}^\bullet) \cong H^1(X, \Omega_{\VV_\alpha}^1) \oplus H^0(X, \Omega^2_{\VV_\alpha}) \cong \left( H^1(X, \Oo_X) \otimes V^*\right) \oplus \bigwedge^2 V^*\ .
\]
Then, once the vector space $V$ is fixed, central extensions of the form \eqref{eq central extension of VV_alpha} are classified by the element defining the UTAI Lie algebroid $\VV_\alpha$, $\alpha \in V^* \otimes \mathfrak{g}$, and the choice of $\HH^2(X, \tau^{\geq 1} \Omega_{\VV_\alpha}^\bullet) \cong (V^* \otimes \hat{\mathfrak{g}}) \oplus \bigwedge^2 V^*$. Two Lie algebroids of the form $\VV_\alpha$ and $\VV_{\alpha'}$, are isomorphic if and only if $\alpha$ and $\alpha'$ are related by the action of $\GL(V)$. This action extends naturally to $(V^* \otimes \hat{\mathfrak{g}}) \oplus \bigwedge^2 V^*$. Therefore, isomorphism classes of pairs $(\VV_\alpha, \widetilde{\VV}_\alpha)$ are classified by $\GL(V)$-orbits of triples $(\alpha, \beta, \gamma)$. The first statement follows then by the equivalence between \ref{it lambda Xi} and \ref{it L wtL} of Lemma \ref{lemma_Lambda-algebroids}.

A extension of Lie algebroids of the form \eqref{eq central extension of VV_alpha} is determined by the extension of the underlying vector bundles and the extension of the bracket. One has the isomorphism $H^1(X, \Oo_X) \otimes V^* \cong \Ext^1(V \otimes \Oo_X, \Oo_X) $ so $\beta \in V^* \otimes \hat{\mathfrak{g}}$ determines the extension of the underlying vector bundles. One can also check that $\gamma \in \bigwedge^2 V^*$ determines the extension of the bracket. Recalling that $c(\Oo_X)$ lies, by construction, in the centre of $\widetilde{\VV}_\alpha$, we have that  
\[
\left[ s , t \right]_{\widetilde{\VV}_\alpha} = 
\left[ f\otimes u, g\otimes v \right]_{\VV_\alpha} + c(fg) \cdot \gamma(u,v),
\]
for any (local) sections $s,t$ of $\widetilde{\VV}_\alpha$ that project, respectively, to $f \otimes u$ and $g \otimes v$, where $f,g \in \Oo_X$ and $u,v \in V$. Therefore, $[\, \cdot \, , \, \cdot \, ]_{\widetilde{\VV}_\alpha}$ is trivial if and only if $\gamma = 0$ and $[\, \cdot \, , \, \cdot \, ]_{\VV_\alpha}$ is trivial, which occurs when $\alpha = 0$.
\end{proof}

After Proposition \ref{pr classification of UTAI D-algebras}, we denote by $\Lambda^\alpha_{\beta,\gamma}$ the $D$-algebra over $X$ associated to the triple $(\alpha,\beta,\gamma)$. When $\beta = \gamma = 0$, the extension of Lie algebroids is trivial $\widetilde{\VV}_\alpha = \VV_\alpha \oplus \Oo_X$, and the associated $D$-algebra is untwisted $\Uu^0(\widetilde{\VV}_\alpha) = \Uu(\VV_\alpha)$. In this case we denote $\Lambda^\alpha = \Lambda^\alpha_{0,0} = \Uu(\VV_\alpha)$. On the other hand, for $\alpha = \gamma = 0$ the $D$-algebra is abelian and we use the notation $\Lambda_\beta = \Lambda^0_{\beta,0}$. We dedicate the rest of the section to give a detailed description of the abelian $D$-algebras $\Lambda_\beta$.


Let $\beta \in V^* \otimes \hat{\mathfrak{g}} = \Ext^1(\Oo_X, \Oo_X) \otimes V^*$, this determines an extension of vector bundles,
\begin{equation} \label{eq transformed Atiyah sequence}
0 \longrightarrow \Oo_X \longrightarrow \Vv^\beta \longrightarrow  \Oo_X \otimes V \longrightarrow 0.
\end{equation}
Then, recalling the definition of the reduced enveloping algebra, one can describe the abelian $\Oo_X$-algebra $\Lambda_\beta$ as follows: 
\[
\Lambda_{\beta} = \Uu^0(\Vv^\beta) \cong \Sym^\bullet(\Vv^\beta)/(\id_{\Vv^\beta} - \id),
\]
where we denote by $\id_{\Vv^\beta}$ the image of $1 \in \Oo_{X}$ in $\Vv^\beta$ and by $\id$ the identity in $\Sym^\bullet(\Vv^\beta)$.

On the other hand, following \cite{mazur-messing}, one also has the isomorphism $V^* \otimes \hat{\mathfrak{g}} \cong \Ext^1(X, V^*)$, so $\beta$ determines an extension of abelian group schemes,
\[
0 \to V^* \to X^\beta \to X \to 0.
\]
Remark that $X^\beta$ is the principal $V^*$-bundle over $X$ associated to the class $\beta \in H^1(\hat{X}, V^*)$;  denote the projection by $\chi: X^\beta \to X$. The following is a generalization of \cite[Proposition 4.4]{rothstein} and only applies for abelian UTAI $D$-algebras.

\begin{proposition}
There is an isomorphism of $\Oo_X$-algebras
\[
\Lambda_{\beta} = \Lambda^0_{\beta, 0} \cong \chi_*(\Oo_{X^{\beta}}).
\]
\end{proposition}

\begin{proof}
Take a $1$-cocycle $(\{ U_i \}, \{\psi_{ij} \} )$ in the class of $\beta$, where $\psi_{ij} \in \Oo_{U_i \cap U_j} \otimes V^*$. Observe that $\Lambda_{\beta}$ is described locally by
\[
\Lambda_{\beta}|_{U_i} \cong \Sym^\bullet (V \otimes \Oo_{U_i})
\]
and, when we restrict to $U_i \cap U_j$, the elements $a \in \Lambda_{\beta}|_{U_i}$ glue with the elements $a' \in \Lambda_{\alpha}|_{U_j}$ if and only if 
\begin{equation} \label{eq gluing condition}
a' = a + \psi_{ij}(a),
\end{equation}
where for $a \in V$, $\psi_{ij}(a)$ is the duality pairing, and one extends it to the symmetric algebra by requiring to respect the product.

Since
\[
\chi_*(\Oo_{X^{\beta}})|_{U_i} \cong \Sym^\bullet (V^{**} \otimes \Oo_{U_i}) \cong \Sym^\bullet (V \otimes \Oo_{U_i}),
\]
we have a natural local isomorphism of algebras between $\Lambda_{\beta}$ and $\chi_*(\Oo_{X^{\beta}})$. This extends to a global isomorphism, since the gluing condition of the $\chi_*(\Oo_{X^{\beta}})|_{U_i}$ is again \eqref{eq gluing condition}.
\end{proof}

Since the projection $\chi$ is an affine morphism, one has that $\chi_*$ gives an equivalence between the categories of $\Oo_{X^\beta}$-modules on $X^\beta$ and $\chi_*(\Oo_{X^\beta})$-modules on $X$. This allow us to describe $\Lambda_\beta$-modules in terms of sheaves over $X^\beta$.

\begin{corollary} \label{co schematic desccription of Lambda_beta-modules}
One has an equivalence of categories 
\[
\Mod(\Lambda_\beta) \cong \Mod(\Oo_{X^\beta}).
\]
\end{corollary}

\subsection{The Polishchuk-Rothstein transform}

Given an abelian variety $X$ of dimension $d$ we denote by $\hat{X}$ its dual. Consider the projections
\[
\xymatrix{
 & X \times \hat{X} \ar[rd]^{p} \ar[ld]_{q} &
\\
X &  & \hat{X}.
}
\]

Let $\Pp$ denote the Poincar\'e bundle on $X \times \hat{X}$ normalized as $\Pp|_{\{ x_0 \} \times \hat{X}} \cong \Oo_{\hat{X}}$. One can construct the {\em Fourier-Mukai functors}
\begin{equation} \label{eq definition of Phi}
\morph{\Dd^b_\coh(X)}{\Dd^b_\coh(\hat{X})}{\Ff}{\R p_*(\Pp \otimes_\L q^* (-1)^* \Ff)}{}{\Phi}
\end{equation}
and 
\begin{equation} \label{eq definition of Psi}
\morph{\Dd^b_\coh(\hat{X})}{\Dd^b_\coh(X)}{\Gg}{\R q_*(\Pp \otimes_\L p^* \Gg).}{}{\Psi}
\end{equation}
It was proved by Mukai \cite{mukai} that this provides an equivalence of categories since $\Phi \circ \Psi = [-d]$ and $\Psi \circ \Phi = [-d]$.

Consider the product $X \times X \times \hat{X} \times \hat{X}$ and denote by $p_{ij}$ the natural projection of the $i$-th and $j$-th factors. Taking
\[
\xymatrix{
 & X \times X \times \hat{X} \times \hat{X}  \ar[dl]^{p_{13}}  \ar[dr]_{p_{24}} &
\\
X \times \hat{X} & & X \times \hat{X},
}
\]
one can construct $\overline{\Pp} = p_{13}^* \Pp \otimes p_{24}^* \Pp$. The dual variety of $X \times X$ is $\hat{X} \times \hat{X}$, and $\overline{\Pp}$ is the Poincar\'e bundle over their product. Then one has the Fourier-Mukai equivalence 
\begin{equation} \label{eq definition of olPhi}
\overline{\Phi} : \Dd(X \times X) \longrightarrow \Dd(\hat{X} \times \hat{X})
\end{equation}
induced by $\overline{\Pp}$. 

We now recall a definition from \cite[Section 6.2]{polishchuk&rothstein}. A $D$-algebra $\Lambda$ on $X$ is {\em special} if $\onproduct{\Lambda}$ is quasi-coherent and has a filtration $0 = \onproduct{\Lambda}_{-1} \subset \onproduct{\Lambda}_0 \subset \onproduct{\Lambda}_1 \subset \dots$ such that $\bigcup_i \onproduct{\Lambda}_i = \onproduct{\Lambda}$ and $\onproduct{\Lambda}_i / \onproduct{\Lambda}_{i-1} \cong \Delta_* \Oo_X^{\oplus n_i}$ for every $i \geq 0$ and some positive number $n_i$, possibly infinite. We remind that $\Delta$ denotes the diagonal morphism. 

By definition, a $D$-algebra $\Lambda$ on $X$ can be endowed with the differential $\Oo_X$-bimodule structure, that is a quasicoherent sheaf on $X \times X$ supported on the diagonal that we denoted by $\onproduct{\Lambda}$. One can apply $\overline{\Phi}$ to $\onproduct{\Lambda}$ to obtain a sheaf on $\hat{X}\times \hat{X}$. By \cite{polishchuk&rothstein}, when $\Lambda$ is special, $\overline{\Phi}(\onproduct{\Lambda})$ is again a differential $\Oo_{\hat{X}}$-bimodule (quasicoherent and supported in the diagonal) and special. Therefore, it defines a special $D$-algebra on $\hat{X}$ that we denote by $\hat{\Lambda}$, the unique $D$-algebra satisfying
\begin{equation} \label{eq definition of hatLambda}
\onproduct{\hat{\Lambda}} = \overline{\Phi}(\onproduct{\Lambda}). 
\end{equation}
We call $\hat{\Lambda}$ the {\em Fourier-Mukai dual} of $\Lambda$. 

\begin{theorem}[\cite{polishchuk&rothstein} Theorem 6.5] \label{tm polishchuk-rothstein transform} 
Let $\Lambda$ be a special $D$-algebra on $X$ and $\hat{\Lambda}$ its Fourier-Mukai dual. Given an element $\Ee \in \Dd^b_\coh(\Mod(\Lambda))$, the Fourier-Mukai transform $\Phi(\Ee)$ has a natural structure of $\hat{\Lambda}$-module. Thus one obtains a functor 
\[
\Phi^\Lambda : \Dd^b_\coh(\Mod(\Lambda)) \stackrel{\cong}{\longrightarrow} \Dd^b_\coh(\Mod(\hat{\Lambda})), 
\]
which is an equivalence of categories, such that the diagram 
\[
\xymatrix{
\Dd^b_\coh(\Mod(\Lambda)) \ar[d]_{f_\Lambda} \ar[r]^{\Phi^\Lambda} & \Dd^b_\coh(\Mod(\hat{\Lambda})) \ar[d]^{f_{\hat{\Lambda}}}
\\
\Dd^b_\coh(X) \ar[r]^{\Phi} & \Dd^b_\coh(\hat{X}),
}
\]
is commutative (where $f_\Lambda$ and $f_{\hat{\Lambda}}$ are the corresponding forgetful functors).
\end{theorem}

Since the associated Lie algebroid $\VV_\alpha$ is UTAI, the $D$-algebras $\Lambda^\alpha_{\beta,\gamma} = \Uu^0(\widetilde{\VV}_\alpha)$ have an underlying vector bundle which is given by extensions of the trivial bundle. Therefore, the $\onproduct{\Lambda^\alpha_{\beta, \gamma}}$ are succesive extensions of $\Delta_*\Oo_X^{\oplus n_i}$ and they are indeed special. In this case we can compute explicitly the Fourier-Mukai dual $D$-algebra. Recalling \eqref{eq definition of g} and \eqref{eq definition of g*}, we observe that UTAI $D$-algebras over $\hat{X}$ are classified by $(\hat{\alpha}, \hat{\beta}, \hat{\gamma}) \in (V^* \otimes \hat{\mathfrak{g}}, V^* \otimes \mathfrak{g}, \bigwedge^2 V^*)$ and we denote by $\hat{\Lambda}^{\hat{\alpha}}_{\hat{\beta}, \hat{\gamma}}$ the corresponding $D$-algebra over $\hat{X}$. We observe that $\alpha, \hat{\beta} \in V^* \otimes \mathfrak{g}$ and $\beta, \hat{\alpha} \in V^* \otimes \hat{\mathfrak{g}}$. 

\begin{lemma} \label{lm Lambda^VV is special}
Let $\Lambda^\alpha_{\beta,\gamma}$ be the $D$-algebra over $X$ associated to the triple $(\alpha,\beta,\gamma)$.

Then $\hat{\Lambda}^{-\beta}_{\alpha,\gamma}$ is the Fourier-Mukai dual of $\Lambda^\alpha_{\beta,\gamma}$.
\end{lemma}

\begin{proof}
According to \eqref{eq definition of hatLambda}, we need to Fourier-Mukai transform $\onproduct{\Lambda^\alpha_{\beta, \gamma}}$, to obtain the corresponding Fourier-Mukai dual $D$-algebra. Recall that the Fourier-Mukai equivalence $\overline{\Phi}$ of \eqref{eq definition of olPhi} sends $\Delta_* \Oo_X$ to $(1\times -1)^*\hat{\Delta}_*  \Oo_{\hat{X}}$; 
therefore, it follows that
\[
\overline{\Phi}\left ( \Delta_*(V \otimes \Oo_X) \right ) \cong (1\times -1)^* \hat{\Delta}_* (V \otimes \Oo_{\hat{X}}).
\]

Moreover one has that (cf. \cite[Proposition 1.34]{ugo&bartocci&ruiperez}) $\overline{\Phi}$ induces an isomorphism
\begin{equation} \label{eq Ext^1 cong hatExt^1}
\Ext^1_{\Oo_{X \times X}}\left ( \Delta_*(V \otimes \Oo_X), \Delta_*\Oo_X \right ) \cong \Ext^1_{\Oo_{\hat{X} \times \hat{X}}}(\hat{\Delta}_* (V \otimes \Oo_{\hat{X}}), \hat{\Delta}_* \Oo_{\hat{X}}). 
\end{equation}

Consider now the algebra $\Lambda^\alpha_{\beta,\gamma}$; by looking at the associated $\Oo_{X\times X}$-module $\onproduct{\Lambda^\alpha_{\beta,\gamma}}$, one obtains the extension
\[
0 \to \Delta_* \Oo_X \to \onproduct{(\Lambda^\alpha_{\beta,\gamma})_1} \to \Delta_* (\Oo_X \otimes V) \to 0
\]
which is classified by the class $(\alpha, \beta) \in \Ext^1_{\Oo_{X \times X}}\left ( \Delta_*(V \otimes \Oo_X), \Delta_*\Oo_X \right )$ thanks to Lemma \ref{lm description of Ext^1_XxX}. After \eqref{eq Ext^1 cong hatExt^1}, we have that $\bar{\Phi}(\onproduct{(\Lambda^\alpha_{\beta,\gamma})_1})$ is the extension associated to the class $(-\beta,\alpha) \in \Ext^1_{\Oo_{\hat{X} \times \hat{X}}}(\hat{\Delta}_* (V \otimes \Oo_{\hat{X}}), \hat{\Delta}_* \Oo_{\hat{X}})$.

To prove that $\gamma$ is sent to itself, one uses the same argument of \cite[Section 4]{rothstein}.
\end{proof}

\section{Semistable $\Lambda$-modules on abelian va\-rie\-ties}
\label{sc moduli spaces}

\subsection{The rank $1$ case}
\label{sc rank 1}

The aim of this section is to give a geometric description of modules for certain $D$-algebras. In order to do so, we will make use of the Polishchuk-Rothstein transform in Theorem \ref{tm polishchuk-rothstein transform} and the equivalence of categories in Corollary \ref{co schematic desccription of Lambda_beta-modules}. Therefore, from now on, we focus on those UTAI $D$-algebras whose Fourier-Mukai dual $D$-algebra is abelian.

The choice of $\alpha \in V^* \otimes \mathfrak{g}$ determines, on the one hand, the UTAI Lie algebroid $\VV_\alpha$ and the associated untwisted $D$-algebra $\Lambda^\alpha = \Uu(\VV_\alpha)$ over $X$. On the other hand, $\alpha$ determines the abelian $D$-algebra $\hat{\Lambda}_\alpha$ over $\hat{X}$ and the extension of abelian group schemes,
\[
0 \longrightarrow V^* \longrightarrow \hat{X}^\alpha \stackrel{\chi}{\longrightarrow} \hat{X} \longrightarrow 0.
\]
After Lemma \ref{lm Lambda^VV is special}, we see that the abelian $D$-algebra $\hat{\Lambda}_\alpha = \hat{\Lambda}^0_{\alpha, 0}$ is the Fourier-Mukai dual of $\Lambda^\alpha = \Lambda^\alpha_{0,0}$.

As is stated in Corollary \ref{co schematic desccription of Lambda_beta-modules}, the affine projection $\chi$ induces an equivalence between the categories of $\Oo_{\hat{X}^\alpha}$-modules on $\hat{X}^\alpha$ and $\hat{\Lambda}_{\alpha} \cong \chi_*(\Oo_{\hat{X}^\alpha})$-modules on $\hat{X}$. Therefore, $R\chi_*$ is a derived equivalence
\begin{equation} \label{eq derived equivalence for modules of affine algebras}
R\chi_*: \Dd^b_\coh(\Mod(\hat{\Lambda}_{\alpha})) \stackrel{\cong}{\longrightarrow} \Dd^b_\coh(\hat{X}^\alpha).
\end{equation}

\begin{corollary} \label{co Polishchuk-Rothstein transform}
The Polishchuk-Rothstein transform of Theorem \ref{tm polishchuk-rothstein transform} combined with \eqref{eq derived equivalence for modules of affine algebras} gives a derived equivalence of categories
\begin{equation} \label{eq Phi^alpha}
\Phi^{\alpha} : \Dd^b_\coh(\Mod(\Lambda^{\alpha})) \stackrel{\cong}{\longrightarrow} \Dd^b_\coh(\hat{X}^\alpha)
\end{equation}
that extends to the relative case. Furthermore, the diagram
\begin{equation} \label{eq diagram Phi^alpha and Phi}
\xymatrix{
\Dd^b_\coh(\Mod(\Lambda^{\alpha})) \ar[d]_{f_{\alpha}} \ar[rr]^{\Phi^{\alpha}} & & \Dd^b_\coh(\hat{X}^\alpha) \ar[d]^{R\chi_*}
\\
\Dd^b_\coh(X) \ar[rr]^{\Phi} & & \Dd^b_\coh(\hat{X}),
}
\end{equation}
commutes, where $f_{\alpha}$ is the forgetful functor. 
\end{corollary}

\begin{remark} \label{rm rank1 Lambda modules correspond to points of X alpha}
The Fourier-Mukai correspondence $\Phi$ sends the complex concentrated on degree $0$ consisting of a locally free sheaf $\Ff_0$ of rank $1$ and trivial Chern classes to a complex concentrated in degree $0$ consisting of a sky-scraper sheaf at the point $\hat{x}_0 \in \hat{X}$ (i.e. such a complex is WIT of index $0$, where we are using the notation of \cite[Definition 2.3]{mukai}). Analogously, the Polishchuk-Rothstein transform $\Phi^\alpha$ sends a complex of $\Lambda^\alpha$-modules, concentrated in degree $0$ and consisting in a topologically trivial $\Lambda^\alpha$-module $(\Ff, \theta)$ of rank $1$, to a complex concentrated in degree $d$ given by the structure of a skyscraper sheaf on $\hat{X}^\alpha$ (i.e. such a complex is WIT of index $d$). Under this equivalence, rank $1$ $\Lambda^\alpha$-modules over $X$ correspond with geometric points of $\hat{X}^\alpha$. 
Since the equivalence of categories $\Phi^\alpha$ extends to the relative case, one gets a correspondence between $\Lambda^\alpha$-modules over $X_S$ and $S$-points of length $1$ of $\hat{X}^\alpha$. 
\end{remark}

In other words, the category $\Mmm^{\sst}_{X}(\Lambda^\alpha,1)(S)$ of (semistable) rank $1$ $\Lambda^\alpha$-modules on $X_S$ with vanishing Chern classes is equivalent to the category of morphisms $\Hhom(S, \hat{X}^\alpha)$.

\begin{corollary} \label{co M_X lambda 1}
One has an isomorphism of stacks
\[
\Mmm^{\sst}_{X}(\Lambda^\alpha,1) \cong \hat{X}^\alpha.
\]
Therefore, the moduli space of (semistable) $\Lambda^{\alpha}$-modules of rank $1$ and vanishing Chern classes is
\[
\M_X(\Lambda^{\alpha}, 1) \cong \hat{X}^\alpha.
\]
\end{corollary}

\subsection{Stability of $\Lambda$-modules}

The main purpose of this section is to describe semistable, stable and polystable $\Lambda$-modules over abelian varieties. This will allow us to study the moduli functor \eqref{eq stack of Lambda-modules}. 

Thanks to the work of \cite{simpson1} one can study the stability of our $\Lambda$-modules in terms of the stability of the underlying $\Oo_X$-module. 

\begin{proposition}[\cite{simpson1} Lemma 3.3] \label{pr mu-semistable Lambda-mod implies mu-semistable O-mod}
Let $(\Ff, \theta)$ be a $\mu$-semistable $\Lambda^{\alpha}_{\beta,\gamma}$-module over the abelian variety $X$. Then $\Ff$ is a $\mu$-semistable $\Oo_X$-module.
\end{proposition}

\begin{proof}
Since $\Gr_1 \Lambda^{\alpha}_{\beta, \gamma}$ is generated by global sections, the statement follows from \cite[Lemma 3.3]{simpson1}.
\end{proof}

\begin{definition}
Let $\Ff$ be a coherent $\Oo_X$-module. A {\em complete} filtration of $\Ff$ is a finite,  increasing filtration $\Ff_i \subseteq \Ff_{i+1} \subseteq \Ff$ of $\Oo_X$-submodule such that for any $i$ the quotient $\Ff_i/\Ff_{i-1}$ is a line bundle on $X$.
\end{definition}

\begin{remark}
In particular, if $\Ff$ admits a complete filtration, then $\Ff$ is a locally free $\Oo_X$-module.
\end{remark}

We have, from \cite{mehta&nori}, the following description of semistable sheaves on abelian varieties with vanishing Chern classes. Note that the description is still valid for $\mu$-semistable sheaves thanks to \cite[Remark 2.6]{mehta&nori}.

\begin{theorem}[\cite{mehta&nori} Theorem 2] \label{tm description of sheaves}
Let $X$ be an abelian variety and $\Ff$ a semistable (resp. $\mu$-semistable) torsion free sheaf on $X$ with $c_i(\Ff) = 0$. Then the Jordan-H\"older filtration of $\Ff$ is a complete filtration.
\end{theorem}

As a consequence, when the conditions of Theorem \ref{tm description of sheaves} hold, the notions of $\mu$-semistability and semistability of sheaves coincide. This allow us to extend Proposition \ref{pr mu-semistable Lambda-mod implies mu-semistable O-mod}.

\begin{corollary} \label{co semistable Lambda-mod implies semistable O-mod}
Let $(\Ff, \theta)$ be a semistable $\Lambda^{\alpha}$-module over the abelian variety $X$. Suppose that $\Ff$ is a torsion free sheaf on it with $c_i(\Ff) = 0$. Then, $\Ff$ is a locally free semistable $\Oo_X$-module.
\end{corollary}

\begin{remark} \label{rm M^LF = M}
As a side consequence of this result, one can prove that the condition $\LF(\xi)$ is superfluous in our case and $\M_{X}(\Lambda^{\alpha}, n)$ can be described by the GIT quotient \eqref{eq description of M as a GIT quotient}.
\end{remark}

Vector bundles with flat $\Lambda^{\DR}$-connections over abelian varieties where studied by Matshushima in \cite{matshushima} who proved that these objects reduce their structure group to the subgroup of upper-diagonal matrices. In particular, one has the following analogue of Theorem \ref{tm description of sheaves}:

\begin{proposition}[\cite{matshushima}] \label{pr filtration for DR}
Let $(\Ff, \nabla)$ be a coherent $\Oo_X$-module with a flat connection on an abelian variety $X$. In particular, $\Ff$ is locally free and the Chern classes of $\Ff$ vanish. Then the Jordan-H\"older filtration of $(\Ff, \nabla)$ is complete.
\end{proposition}

In the rest of the section we prove that every semistable $\Lambda^\alpha$-module with va\-nishing Chern classes has a complete Jordan-H\"older filtration. This implies that the only stable $\Lambda^\alpha$-modules are of rank $1$. The strategy is to first show that this is true when $\alpha = 0$, and deduce the general case from this and the result of Matsushima.

\begin{lemma}  \label{lemma_rank_one_invariant}
Let $\Ff$ be a vector bundle on an abelian variety $X$, semistable and with trivial Chern classes. Let $\phi\in H^0(X, \End(\Ff))$ be an endomorphism of $\Ff$.

Then there exist a rank $1$ $\phi$-invariant subbundle of $\Ff$ of degree equal to zero.
\end{lemma}
\begin{proof}
Consider $\im \phi$. This is both a subbundle and a quotient of $\Ff$. Then, by semistability of $\Ff$, 
$$
p(\Ff) \leq p(\im \phi) \leq p(\Ff)
$$
so $p(\im \phi) = p(\Ff)$. Since the tangent bundle of $X$ is trivial and $c_i(\Ff) = 0$, one deduces that $\im \Ff$ has trivial Chern classes. 

Now assume that $\ker \phi$ is non zero. Since $\im \phi$ has trivial Chern classes, the Chern classes of $\ker \phi$ vanish as well. Moreover, as a subbundle of a semistable bundle, it must be semistable. Then the Jordan-H\"older filtration of $\ker \phi$ is complete, and the first element of the filtration is a $\phi$-invariant line subbundle of $\Ff$ of rank $1$ and degree zero.

If $\ker \phi$ is equal to zero, we can apply the same arguments to the endomorphism $\phi - \lambda 1$, and for $\lambda$ an eigenvalue of $\phi$ we obtain a non trivial $\phi$-invariant line-subbundle of $\Ff$ of degree zero. Since $\phi$ has at least one eigenvalue, the proof is complete.
\end{proof}

\begin{corollary} \label{pr filtration for Dol}
Let $K$ be a vector space, and consider $\Lambda^0 \cong \Sym^\bullet (\Oo_X \otimes K)$. Let $(\Ff,\theta)$ be a semistable $\Lambda^0$-module with vanishing Chern classes. Then $\Ff$ has a $\theta$-invariant rank $1$ subsheaf of degree equal to zero. 
\end{corollary}

\begin{proof}
Take the Lie algebroid $\KK = (\Oo_X \otimes K, 0, [ \cdot , \cdot ])$ defined by the vector space $K$, a trivial Lie bracket and anchor equal to the $0$ morphism. Note that, due to Corollary \ref{co semistable Lambda-mod implies semistable O-mod} and Theorem \ref{tm description of sheaves}, $\Ff$ is a semistable vector bundle with vanishing Chern classes, and $\theta$ a flat $\KK$-connection on it.

After fixing a base $b_1,\ldots,b_m$ of $K$, we can regard $\theta$ as a family of $m$ commuting endomorphisms $\theta_i$ of $\Ff$ by defining $\theta_i = \theta_{b_i}$. Since these endomorphisms commute, we can find a subbundle $\Ff'$ of $\Ff$ on which each $\theta_i$ acts as multiplication by a constant, so that any subspace of $\Ff'$ is $\theta$-invariant. By iterating the arguments in the proof of Lemma \ref{lemma_rank_one_invariant}, one shows that $\Ff'$ is semistable and has Chern classes equal to zero. Then the Jordan-H\"older filtration of $\Ff'$, is complete and $\theta$-invariant, and its first term is a rank $1$ subsheaf of degree equal to zero.
\end{proof}

\begin{proposition} \label{pr stable implies rank 1}
Let $(\Ff, \theta)$ be a semistable $\Lambda^{\alpha}$-module over the abelian variety $X$, and assume that $\Ff$ has trivial Chern classes. 

Then, there exist a rank $1$ $\theta$-invariant subsheaf of $\Ff$ with trivial Chern classes. In particular, the Jordan-H\"older filtration of $(\Ff, \theta)$ is complete.
\end{proposition}

\begin{proof}

Recall the notation of Section \ref{sc some lie algebroids}: $K = \ker (\alpha)$, $F  = \im(\alpha)$ and $N = \coker(\alpha) $, as well as the Lie algebroids $\KK$ and $\FF$.

First of all, remark that by composing $\theta$ with the dual of the inclusion $\KK \hookrightarrow \VV_\alpha$, one obtains $\theta_{\KK}: \Ff \to \Ff \otimes \Omega^1_{\KK}$ that defines a $\KK$-module structure on $\Ff$. We now use some arguments from the proof of Corollary \ref{pr filtration for Dol} to understand the endomorphism $\theta_{\KK}$: let $b_1,\ldots, b_m$ be a base of $K$, and $\theta_i$ the associated endomorphisms of $\Ff$. For any $\lambda = (\lambda_1,\ldots, \lambda_m)$ $m$-tuple of elements of $\CC$, consider the eigenbundle 
$$
\Ff^{\lambda} = \{ s\in \Ff \ |\ \theta_i s = \lambda_i s \ \text{for all } i\ \}
$$ 
For any $\lambda$, the subbundle $\Ff^\lambda$ is $\theta_{\KK}$-invariant, and any of its subbundle is $\theta_{\KK}$-invariant.

Choose one $\lambda$ such that $\Ff^\lambda$ is different from zero. Then $\Ff^\lambda$ is also $\theta$-invariant, since from the flatness of $\theta$ follows that for any $v$, section of $\VV_\alpha$, and any $s$, section of $\Ff'$, 
$$
\theta_i \theta_v s = \theta_v \theta_i s + \theta_{[1 \otimes b_i,v]} s = \lambda_i  \theta_v s\ ,
$$
where the last equality follows from $[1\otimes b_i , v] = v(1)\cdot b_i = 0$, since $1$ is a constant function.

Fix a splitting $s: F \to V$; this induces an isomorphism $\VV_\alpha \cong \FF \oplus \KK$, and  accordingly we can split $\theta$ as $\theta = \theta_{\KK} \oplus \theta_\FF$, where $\theta_\FF: \Ff \to \Ff \otimes \Omega^1_{\FF}$ is defined by composing $\theta$ with the dual of $s$. Since the induced splitting $1\otimes s : \FF \to \VV_\alpha$ is a Lie algebroid morphism, $\theta_\FF$ is a flat $\FF$-connection on $\Ff$, and $\Ff^\lambda$ is $\theta_\FF$-invariant.

Let us fix a splitting of the sequence $0\to F \to \mathfrak{g} \to N \to 0$ and consider the associated projection $t: \mathfrak{g} \to F$. By composing $\theta_\FF$ with the dual of $t$, one obtains $\nabla: \Ff \to \Ff \otimes \Omega^1_X$ a $\Tt X$-connection. Since the splitting $1 \otimes t: \Tt X \to \FF$ is a Lie algebroid morphism, $\nabla$ is a flat $\Tt X$-connection on $X$, and $\Ff^\lambda$ is again $\nabla$-invariant. Now we can choose a complete Jordan-H\"older filtration of $(\Ff^\lambda,\nabla)$ guaranteed to exists by Proposition \ref{pr filtration for DR}, and we denote by $\Ff_1$ its first term, that is a rank $1$ $\nabla$-invariant subsheaf with trivial Chern classes. Then, since $\Omega^1_X \to \Omega^1_\FF$ is surjective, $\Ff_1$ is also $\theta_\FF$-invariant; since $\Ff_1 \subseteq \Ff^\lambda$, it is also $\theta_{\KK}$-invariant, thus it is $\theta$-invariant, and the proof is complete.
\end{proof}

\subsection{The moduli of semistable $\Lambda$-modules}

In this section we describe $\Mmm^{\sst}_X(\Lambda^\alpha)$ and $\M_X(\Lambda^\alpha)$. For that, we need to consider the stack of torsion sheaves on a scheme $Y$,
\[
\Ttt(Z, n) \, : \, \Aff^{\op} \longrightarrow \Groupoids
\]
associating to each scheme $S$ the category of coherent sheaves $\Gg \to Y \times S$ with $\pi: \supp(\Gg) \to S$ finite and for all $s \in S$, the Hilbert polynomial is $\Pp_{\Gg_{s}} = n$. We denote the associated moduli functor by
\begin{equation} \label{eq stack of torsion-free}
\quotient{\Ttt(Z, n)}{\! \cong} \, : \, \Aff^{\op} \longrightarrow \Sets,
\end{equation}
obtained by considering the isomorphism classes of $\Ttt(Z,n)$.

We now address the proof of the first main result of the paper, from which Theorem \ref{tm description of M_X} follows naturally. Groechenig \cite[Lemma 4.6 and Proposition 4.10]{groechenig} proved it for Higgs bundles and vector bundles with flat connections over elliptic curves.

\begin{theorem} \label{tm Mmm = Ttt}
There exists an isomorphism of stacks
\[
\Mmm^{\sst}_{X}(\Lambda^\alpha, n) \cong \Ttt(\hat{X}^\alpha, n),
\]
given by the equivalence of categories $\Phi^{\alpha}$ from Corollary \ref{co Polishchuk-Rothstein transform}.
\end{theorem}

\begin{proof}
Recall from Remark \ref{rm rank1 Lambda modules correspond to points of X alpha} that rank $1$ (stable) $\Lambda^{\alpha}$-mo\-du\-les on $X$ are WIT (in the sense of \cite{mukai}) and correspond to sky-scraper sheaves on $\hat{X}^\alpha$. By Proposition \ref{pr stable implies rank 1}, every semistable $\Lambda^\alpha$-module $(\Ff_n, \theta_n)$ of rank $n$ and trivial Chern classes on $X$ has a filtration
\[
0 = (\Ff_0, \theta_0) \subset (\Ff_1, \theta_1) \subset \dots \subset (\Ff_{n-1}, \theta_{n-1}) \subset (\Ff_n, \theta_n),
\]
whose successive quotients $(\Ee_i, \theta'_i) = (\Ff_i, \theta_i) / (\Ff_{i-1}, \theta_{i-1})$ are (stable) rank $1$ $\Lambda^\alpha$-modules of trivial Chern classes. Assume, by induction on the rank, that any $(\Ff, \theta)$ with rank $< n$, is WIT, i.e. $\Phi^\alpha(\Ff, \theta)^\bullet$ is a complex supported on degree $d$ consisting on a length $n-1$ sheaf on $\hat{X}^\alpha$. Then, applying $\Phi^\alpha$ to
\[
0 \longrightarrow (\Ff_{n-1}, \theta_{n-1}) \longrightarrow (\Ff_n, \theta_n) \longrightarrow (\Ee_n, \theta'_n) \longrightarrow 0,
\]
one gets the distinguished triangle in $\Dd^b_\coh(\hat{X}^\alpha)$
\[
\Phi^\alpha(\Ff_{n-1}, \theta_{n-1})^\bullet \longrightarrow \Phi^\alpha(\Ff_n, \theta_n)^\bullet \longrightarrow \Phi^\alpha(\Ee_n, \overline{\theta}_n)^\bullet \stackrel{[1]}{\longrightarrow} ,
\]
where $\Phi^\alpha(\Ff_{n-1}, \theta_{n-1})^\bullet$ and $\Phi^\alpha(\Ee_n, \theta'_n)^\bullet$ are complexes supported on degree $d$ consisting on sheaves over $\hat{X}^\alpha$, of lengths $n-1$ and $1$ respectively. It follows that $(\Ff_n, \theta_n)$ is WIT and a length $n$ sheaf on $\hat{X}^\alpha$. Therefore, $\Phi^\alpha$ gives an equivalence of categories between the category $\Mmm^{\sst}_{X}(\Lambda^\alpha, n)(\Spec(\CC))$ of semistable $\Lambda^\alpha$-modules of rank $n$ and trivial Chern classes, and the category $\Ttt(\hat{X}^\alpha, n)(\Spec(\CC))$ of length $n$ sheaves on $\hat{X}^\alpha$.  

Next, we recall that $\Phi^{\alpha}$ is functorial and it is well defined for the trivial $S$-schemes $X_S = X \times S$, $\hat{X}_S = \hat{X} \times S$ and $\hat{X}^\alpha_S = \hat{X}^\alpha \times S$. By \cite[Lemma 3.31]{huybrechts} if the restriction of a bounded complex to each $s \in S$ is WIT, the complex itself is WIT. It follows that every semistable $\Lambda^{\alpha}$-module over $X_S$ with trivial Chern classes is WIT and, hence, one gets an equivalence of categories between $\Mmm^{\sst}_{X}(\Lambda^\alpha, n)(S)$ and $\Ttt(\hat{X}^\alpha, n)(S)$.
\end{proof}

Since we want to describe the moduli space of semistable $\Lambda^\alpha$-modules we need to study the moduli functor associated to $\Ttt(Z,n)$. It is well known that $\Sym^n(Z)$ corepresents $\Ttt(Z,n)/\! \cong$, although we could not find a good reference in the literature. For the sake of completion, we provide a proof of it. We first study the jump phenomena occurring in this moduli problem.

\begin{lemma} \label{lm S-equivalence for lenght n sheaves}
Any torsion sheaf $\Gg$ on $Z$ of finite length is a successive extension of skyscraper sheaves.

Hence, there exists a family $\Gg_{\AA^1} \in \Ttt(Z,n)(\AA^1)$ such that for $t \neq 0$ one has $\Gg_{\AA^1|\{t\} \times Z}$ isomorphic to $\Gg$, and $\Gg_{\AA^1|\{0\} \times Z}$ is isomorphic to the direct sum of the skyscraper sheaves.
\end{lemma}

\begin{proof}
To prove the first part of the lemma, it suffices to show that given a torsion sheaf $\Gg$ supported at a point $z$, there exists always a non-zero morphism $k_z \to \Gg$. Remark that since $\Gg$ is supported at $z$, any morphism $\Oo_Z \to \Gg$ factors through $k_z \to \Gg$, so it suffices to show a non-zero morphism $\Oo_Z\to \Gg$. Now, since $\Gg$ is supported at $z$, the problem is local around $z$, and since $\Gg$ is coherent, locally around $z$ we have a surjective morphism $\Oo_Z^{\oplus a} \to \Gg$, and we conclude the proof of the first statement.

For the second part, note that, by the previous paragraph, there exists a decreasing filtration $\Gg^i$ of $\Gg$ such that the quotients $\Gg^i/\Gg^{i+1}$ are skyscraper sheaves. Then, one constructs the Rees object $\Gg_{\AA^1}$, which is defined as the $\Oo_{Z \times \AA^1}$-submodule of $p_Z^*\Gg \otimes \Oo_{\AA^1\setminus \{0\}}$ generated by $t^{-i} \Gg^i$, where $t$ denotes the coordinate on $\AA^1$. The Rees module satisfies the condition of the lemma.
\end{proof}

\begin{proposition} \label{pr Sym^n corepresents Ttt}
The symmetric product $\Sym^n(Z)$ is a coarse moduli space for the classification of length $n$ torsion sheaves on $Z$ (equivalently, $\Sym^n(Z)$ corepresents \eqref{eq stack of torsion-free}).
\end{proposition}

\begin{proof}
We need to define a natural transformation,
\[
\Theta: \quotient{\Ttt(Z, n)}{\! \cong} \to \Hhom(\bullet, \Sym^n(Z)),
\]
and show that for any other pair $(Y', \Theta')$, where
\[
\Theta' : \quotient{\Ttt(Z, n)}{\! \cong} \longrightarrow \Hhom(\bullet, Y'),
\]
there exists a morphism $f : \Sym^n(Z) \to Y'$ such that $\Theta' = f \circ \Theta$.

To define $\Theta$, let $S$ be an affine scheme, and $\Gg$ a family in $\Ttt(Z,n)(S)/\!\cong$. For any $s\in S$, the sheaf $\Gg_s$ is a torsion sheaf on $Z$ of length $n$, and we can consider its weighted support, $\sum_{z \in Z} \text{length}_z(\Gg_s) \cdot [z]$, which is an element in $\Sym^n(Z)$. Then define $\Theta(\Gg)$ to be the morphism that takes $s\in S$ to the weighted support of $\Gg_s$. 

Let us also construct the family $\Ff \to Z \times \Sym^n(Z)$, whose restriction to the slice associated to $p = [(z_1, \dots, z_n)]_{\sym_n} \in \Sym^n(Z)$ is
\[
\Ff|_{Z \times \{ p \}} \cong \CC_{z_1} \oplus \dots \oplus \CC_{z_n}.
\]
For any $(Y',\Theta')$ as before, take $f := \Theta'(\Ff)$ to be the induced morphism by this family
\[
f: \Sym^n(Z) \longrightarrow Y'.
\]
The fact that $\Theta' = f \circ \Theta$ is a consequence of Lemma \ref{lm S-equivalence for lenght n sheaves}.
\end{proof}

Finally, we proof the first theorem stated in Section \ref{sc introduction}.

\begin{proof}[Proof of Theorem \ref{tm description of M_X}]
The result follows naturally from Theorem \ref{tm Mmm = Ttt} and Proposition \ref{pr Sym^n corepresents Ttt}.
\end{proof}

\section{Marked $\Lambda$-modules}
\label{sc marked}

\subsection{Moduli spaces of marked $\Lambda$-modules}

Consider a smooth projective sche\-me $Y$ with a point $y_0 \in Y$ fixed on it. A {\em $\Lambda$-module over $Y$ marked at $y_0$} is a triple $(\Ff, \theta, \sigma)$, where $(\Ff, \theta)$ is a $\Lambda$-module on $Y$ such that $\Ff$ is a locally free $\Oo_Y$-module and $\sigma \in \Ff|_{y_0}$ is a non-zero element of the fibre of $\Ff$ over $y_0$. Two $\Lambda$-modules marked at $y_0$ $(\Ff, \theta, \sigma)$ and $(\Ff', \theta', \sigma')$ are isomorphic if there is an isomorphism of $\Lambda$-modules $\psi : (\Ff, \theta) \stackrel{\cong}{\to} (\Ff', \theta')$ such that $\psi(\sigma) = \sigma'$. In the relative case, a {\em $\Lambda$-module on $Y_S$ marked at $\xi$} is a triple $(\Ff, \theta, \sigma)$, where $(\Ff, \theta)$ is a $\Lambda$-module on $Y_S$ satisfying the property $\LF(\xi)$ and $\sigma \in H^0(S,\xi^*\Ff)$. 

We say that a $\Lambda$-module over $Y$ marked at $y_0$, $(\Ff, \theta, \sigma)$, is {\em stable} if $(\Ff, \theta)$ is a semistable $\Lambda$-module and there are no $\Lambda$-submodules $\Ff'$ with the same reduced Hilbert polynomial of $\Ff$ and such that $\sigma \in \Ff'|_{y_0}$. In the relative case, a $\Lambda$-module $(\Ff,\theta, \sigma)$ over $Y_S$ is {\em stable} if $(\Ff_s,\theta_s, \sigma_s)$ is stable over $Y_s = Y \times \{ s \}$ for every closed point $s \in S$. We observe that stability implies triviality of the automorphism group.

\begin{lemma} \label{lm stable marked implies simple}
The automorphism group of a stable $\Lambda$-module over $Y$ marked at $y_0$ is trivial.
\end{lemma}

\begin{proof}
Given $(\Ff, \theta, \sigma)$ stable, suppose that there exists a non-trivial automorphism $\psi \in \Aut(\Ff, \theta, \sigma)$. Take the endomorphism $(\psi - \id_\Ff)$, and note that $\Ff' := \ker (\psi - \id_\Ff)$ is a subbundle of $\Ff$ preserved by $\theta$ and such that $\sigma \in \Ff'|_{y_0}$. Observe that $\im (\psi - \id_\Ff)$ and $\ker (\psi - \id_\Ff)$ are both subbundles of $\Ff$, the semistability of $\Ff$ implies that both have vanishing Chern classes, so $c_i(\Ff') = 0$. Since $0 \neq \sigma \in \Ff'|_{y_0}$, we have that $\Ff' \neq 0$. On the other hand, $\Ff' = \ker(\psi - \id_\Ff) \neq \Ff$ since $\psi$ is non-trivial by assumption. Therefore, $\Ff'$ is a proper subbundle contradicting the stability of $(\Ff, \theta, \sigma)$ and proving the claim.
\end{proof}

Let us consider the moduli stack of stable $\Lambda$-modules marked at $y_0$,
\[
\Nnn^{\st}_{Y,y_0}(\Lambda,n) \, : \, \Aff^{\op} \longrightarrow \Groupoids,
\]
sending the scheme $S$ to the category of stable $\Lambda$-modules on $Y_S$ marked at $\xi$, with rank $n$ and vanishing Chern classes. We also consider the associated moduli functor
\[
\Nnn^{\st}_{Y,y_0}(\Lambda,n)/\! \cong \, \, : \, \Aff^{\op} \longrightarrow \Sets,
%
\]
obtained by taking the set of isomorphism classes of a given groupoid.

We can provide a result analogous to Theorem \ref{tm existence of moduli space}.

\begin{theorem} \label{tm GIT description of N}
There exists a coarse moduli space $\N_{Y,y_0}(\Lambda, n)$ for the classification of stable $\Lambda$-modules marked at $y_0$ (equivalently, $\N_{Y,y_0}(\Lambda, n)$ corepresents the moduli functor $\Nnn^{\st}_{Y,y_0}(\Lambda, n)/\!\cong$). The geo\-me\-tric points of $\N_{Y,y_0}(\Lambda, n)$ represent isomorphism classes of stable $\Lambda$-modules marked at $y_0$.
\end{theorem}

\begin{proof}
Let us recall the discussion following Theorem \ref{tm existence of moduli space} and the notation used there. Consider $r \gg 0$ big enough and let $\Quot(V \otimes \Ww, n)^{\LF(\xi)}$ be the subset of $\Quot(V \otimes \Ww, n)$ of those triples $(\Ff, \theta, a : H^0(Y, \Ff(r)) \stackrel{\cong}{\to} V)$ such that $\Ff$ satisfies condition $\LF(\xi)$, it can be proved (see \cite{simpson1}) that it is an open subset. Restricted to this open subset, the universal family $\Uu \to Y \times \Quot(V \otimes \Ww, n)^{\LF(\xi)}$ is locally free along the section $\xi : \Quot(V \otimes \Ww, n)^{\LF(\xi)} \to Y \times \Quot(V \otimes \Ww, n)^{\LF(\xi)}$. Then, one can consider the vector bundle $\xi^* \Uu$ over $\Quot(V \otimes \Ww, n)^{\LF(\xi)}$. We consider the open subset $\widetilde{Q} = \xi^*\Uu - \left ( \Quot(V \otimes \Ww, n)^{\LF(\xi)} \times \{ 0 \} \right )$ given by the substraction of the zero section. Since $\xi^*\Uu$ is a vector bundle, the projection $\pi: \widetilde{Q} \to \Quot(V \otimes \Ww, n)^{\LF(\xi)}$ is an affine morphism. Note that $\GL(V)$ acts on $\widetilde{Q}$ and that, contrary to what happens on its base, the action of its centre is not trivial. Recall that $\Lll_m$, for $m$ big enough, is an ample line bundle with a linearized action of $\SL(V)$ and $\GL(V)$, and note that the centre of $\GL(V)$ acts non-trivially on its fibres. Its pull-back $\widetilde{\Lll}_m = \pi^* \Lll_m$ is a linearization of $\widetilde{Q}$ for the action of $\SL(V)$ and $\GL(V)$. 

We can define the family $\widetilde{\Uu} = (\id_Y \times \pi)^*\Uu \to Y \times \widetilde{Q}$ and the tautological section $\tau : Y \times \widetilde{Q} \to (\id_Y \times \pi)^*\Uu$. It follows, by the universality of $\Uu$, that $(\widetilde{\Uu}, \tau) \to Y \times \widetilde{Q}$ is a universal family for the classification of tuples $(\Ff, \theta, a, \sigma)$, where $(\Ff, \theta)$ is a $\Lambda$-module of rank $n$ and vanishing Chern classes such that $\Ff$ is locally free, $\sigma$ is an element of the fibre $\Ff|_{y_0}$ and $a$ is an isomorphism of vector spaces $H^0(Y, \Ff(r)) \stackrel{\cong}{\to} V$. Observe that two points of $\widetilde{Q}$ parametrize isomorphic marked $\Lambda$-modules if and only if they are related by the action of $\GL(V)$. By the universality of $(\widetilde{\Uu}, \tau)$ in the classification of marked quotient sheaves, one has that $(\widetilde{\Uu}, \tau) \to Y \times \widetilde{Q}$ has the local universal property for the classification problem of $\Lambda$-modules marked at $y_0$. Recall that semistable $\Lambda$-modules are bounded by \cite[Proposition 3.5]{simpson1}. Since in the definition of stable marked $\Lambda$-module we require the semistability of the underlying $\Lambda$-module, the boundedness of the first type of objects follows from the boundedness of the later. Therefore, the GIT quotient $\N_{Y, y_0}(\Lambda, n) = \widetilde{Q} \Slash \GL(V)$, constructed with respect to the linearization $\widetilde{\Lll}_m$, corepresents the moduli functor $\Nnn^{\st}_{Y, y_0}(\Lambda, n)/\! \cong$.

We claim now that, for $m$ big enough, the set of points that are semistable for the action of $\GL(V)$ with respect to $\widetilde{\Lll}_m$, $\widetilde{Q}^{sst}_{\widetilde{\Lll}_m}$, coincides with the subset of $\widetilde{Q}$ given by the points whose associated marked $\Lambda$-modules are stable. We start by observing that, given any $\widetilde{q}_1 \in \widetilde{Q}^{sst}_{\widetilde{\Lll}_m}$, its projection $q_1 = \pi(\widetilde{q}_1)$, is semistable for the action of $\SL(V)$ with respect to $\Lll_m$. By the semistability of $\widetilde{q}_1$, for every $\widetilde{\ell}_1 \in \widetilde{\Lll}_m|_{\widetilde{q}_1} - \{ 0 \}$ one has that the closure of $\GL(V) \cdot (\widetilde{q}_1, \widetilde{\ell}_1)$ does not meet the zero section of $\widetilde{\Lll}_m$. Recall that $\widetilde{\Lll}_m = \pi^*\Lll_m$ and take $\ell_1 \in \Lll_m |_q$ associated to $\widetilde{\ell}_1$. Since $\SL(V)$ is closed inside $\GL(V)$, the closure of $\SL(V) \cdot (\widetilde{q}_1, \widetilde{\ell}_1)$ does not meet the zero section, and, therefore, neither does the closure of $\SL(V) \cdot (q_1, \ell_1)$. Then, as we said, $q_1$ is semistable for the action of $\SL(V)$ with respect to $\Lll_m$, so it determines a semistable $\Lambda$-module. Next, we consider $\widetilde{q}_2 \in \widetilde{Q}$ associated to a tuple $(\Ff_2,\theta_2, a_2, \sigma_2)$ such that the marked $\Lambda$-module $(\Ff_2, \theta_2, \sigma_2)$ is not stable but $(\Ff_2, \theta_2)$ is semistable. We claim that $\widetilde{q}_2$ is unstable for the $\GL(V)$-action with respect to $\widetilde{\Lll}_m$ in this case. Since $(\Ff_2, \theta_2, \sigma_2)$ is not stable but $(\Ff_2, \theta_2)$ is semistable, there exists a subbundle $\Ff'_2 \subset \Ff_2$ preserved by $\theta_2$, with vanishing Chern classes $c_i(\Ff'_2) = 0$, and containing the marking $\sigma_2 \in \Ff'_2|_{x_0}$. Then, $(\Ff_2,\theta_2)$ is $\Ss$-equivalent to $(\Ff'_2, \theta'_2) \oplus (\Ff''_2, \theta''_2)$ and so, in the closure of the $\GL(V)$-orbit of $(\widetilde{q}_2, \widetilde{\ell}_2)$, where $\widetilde{\ell}_2 \in \widetilde{\Lll}_m|_{\widetilde{q}_2} - \{ 0 \}$,  there is a point $(\widetilde{q}'_2, \widetilde{\ell}'_2)$ such that $\widetilde{q}'_2$ determines the marked $\Lambda$-module $(\Ff'_2, \theta'_2, \sigma_2) \oplus (\Ff''_2, \theta''_2, 0)$. We take $V'$ and $V''$ to be, respectively, the images in $V$ of $H^0(Y, \Ff'_2(r))$ and $H^0(Y, \Ff''_2(r))$. Note that $V = V' \oplus V''$ and take the $1$-parameter $\lambda : \CC^* \to \GL(V)$,
\[
\lambda(t) = \begin{pmatrix}
\id_{V'} &  \\ 
 & t \cdot \id_{V''}
\end{pmatrix}.
\]
We observe that the action of $\lambda(t)$ preserves $\widetilde{q}'_2$, but sends $\widetilde{\ell}'_2 \to 0$ as $t \to 0$. Therefore, $\widetilde{q}_2$ is unstable as we anticipated. This shows that $\widetilde{q} = (\Ff, \theta, a, \sigma)$ lies in $\widetilde{Q}^{sst}_{\widetilde{\Lll}_m}$ only if $(\Ff, \theta, \sigma)$ is stable. We prove next the "if" implication. Take $\widetilde{q}_3$ associated to some $(\Ff_3, \theta_3, a_3, \sigma_3)$ with $(\Ff_3, \theta_3, \sigma_3)$ is stable. If $g \in \GL(V)$ is an element of $\Stab_{\GL(V)}(\widetilde{q}_3)$, then there exists $\psi \in \Aut(\Ff_3, \theta_3, \sigma_3)$ satisfying 
\[
(\Ff_3, \theta_3, g \circ a_3, \sigma_3) = (\psi(\Ff_3), \psi(\theta_3), a_3 \circ \psi, \psi(\sigma_3)).
\]
By Lemma \ref{lm stable marked implies simple} and the stability of $(\Ff_3, \theta_3, \sigma_3)$ we have that $\psi = \id_{\Ff_3}$. Therefore $g = \id_V$ and the $\GL(V)$-stabilizer of $\widetilde{q}_3$ is trivial. Suppose now that $\widetilde{q}_3$ is not semistable for the action of $\GL(V)$ with respect to $\widetilde{\Lll}_m$, then there exists $\widetilde{\ell}_3 \in \widetilde{\Lll}_m|_{\widetilde{q}_3} - \{ 0 \}$ such that the closure of the orbit $\GL(V) \cdot (\widetilde{q}_3, \widetilde{\ell}_3)$ intersects the zero section of $\widetilde{\Lll}_m$. Since $\widetilde{q}_3$ has trivial stabilizer, the orbit $\GL(V) \cdot (\widetilde{q}_3, \widetilde{\ell}_3)$ is closed, so there exists $(\widetilde{q}'_3, \widetilde{\ell}'_3)$ in the orbit $\GL(V) \cdot (\widetilde{q}_3, \widetilde{\ell}_3)$ intersecting the zero section. Then, all the elements in $Z_{\GL(V)}(\GL(V)) \cdot (\widetilde{q}_3, \widetilde{\ell}_3)$ intersect the zero section of $\widetilde{\Lll}_m$ and this implies that there exists a point in the orbit $\SL(V) \cdot (\widetilde{q}_3, \widetilde{\ell}_3)$ intersecting the zero section of $\widetilde{\Lll}_m$. As a consequence, the projection $q_3 = \pi(\widetilde{q}_3)$ intersects the zero section of $\Lll_m$ implying that $q_3$ is not a semistable point for the action of $\SL(V)$ with respect to $\Lll_m$ and therefore $(\Ff_3, \theta_3)$ is not semistable. This contradicts the stability of $(\Ff_3, \theta_3, \sigma_3)$ so $\widetilde{q}_3$ is, forcely, semistable for the action of $\GL(V)$ with respect to the linearization $\widetilde{\Lll}_m$. This finish the identification of $\widetilde{Q}^{sst}_{\widetilde{\Lll}_m}$ with the set of points giving a stable marked $\Lambda$-module.

Finally, observe that semistability equals stability, as every point of $\widetilde{Q}^{sst}_{\widetilde{\Lll}_m}$ has trivial stabilizer and closed orbits. Since $\widetilde{Q}^{sst}_{\widetilde{\Lll}_m} = \widetilde{Q}^{st}_{\widetilde{\Lll}_m}$, each point of the quotient represents a $\GL(V)$-orbit.
\end{proof}

In accordance with the definition of framed $\Lambda$-modules, we say that a {\em framed marked $\Lambda$-module} is a tuple $(\Ff, \theta, \varphi, \sigma)$, that is, a marked $\Lambda$-module together with a framing $\varphi : \Ff|_{x_0} \stackrel{\cong}{\to} \CC^n$. Note that in the relative case a framing is an isomorphism $\varphi : \xi^*\Ff \stackrel{\cong}{\to} \Oo_S^{\oplus n}$. An {\em isomorphism} of framed marked $\Lambda$-modules,
\[
f : (\Ff_1, \theta_1, \varphi_1, \sigma_1) \to (\Ff_2, \theta_2, \varphi_2, \sigma_2),
\]
is an isomorphism of framed $\Lambda$-modules such that $f(\sigma_1) = \sigma_2$. We can study the classification of these objects.

\begin{theorem} \label{tm marked R}
The functor that associates to every scheme $S$ the set of isomorphism classes of stable framed marked $\Lambda$-modules on $Y \times S$ is represented by a scheme $\widetilde{\R}_{Y}(\Lambda, n, x_0)$ (equivalently $\widetilde{\R}_{Y}(\Lambda, n, x_0)$ is a fine moduli space for the classification of stable framed marked $\Lambda$-modules). One has a natural morphism
\[
r : \widetilde{\R}_{Y}(\Lambda, n, x_0) \longrightarrow \R_{Y}(\Lambda, n, x_0),
\]
given by forgetting the marking.

There is a natural action of $\GL(n,\CC)$ on $\widetilde{\R}_{Y}(\Lambda, n, x_0)$ for which one can construct a linearization $\widetilde{L}$ such that every point of $\widetilde{\R}_{Y}(\Lambda, n, x_0)$ is stable. The associated GIT quotient is isomorphic to $\N_{Y,y_0}(\Lambda, n)$,
\begin{equation} \label{eq description of N as a GIT quotient}
\N_{Y,y_0}(\Lambda, n) \cong \GIT{\widetilde{\R}_{Y}(\Lambda, n, y_0)}{\GL(n,\CC).}
\end{equation}

Finally, $r$ induces
\[
\rho: \N_{Y,y_0}(\Lambda, n) \longrightarrow \M^{\LF(\xi)}_Y(\Lambda, n).
\]
\end{theorem}

\begin{proof}
Using the framing, one gets a canonical identification between $\Ff|_{x_0}$ and $\CC^n$. Recall from Theorem \ref{tm M=R/GL} that $\R_{Y}(\Lambda, n, y_0)$ classifies all the semistable framed $\Lambda$-modules $(\Ff,\theta,\varphi)$. Then, it is clear that $\R_Y(\Lambda, n, y_0) \times (\CC^n - \{ 0 \})$ classifies all framed marked $\Lambda$-modules whose underlying $\Lambda$-module is semistable. We denote
\[
\widetilde{\R}_{Y}(\Lambda, n, x_0) := \left ( \R_Y(\Lambda, n, y_0) \times (\CC^n - \{ 0 \}) \right )^{st},
\]
the open subset given by those marked framed $\Lambda$-modules whose underlying marked $\Lambda$-module is stable. Denote by $r$ the obvious projection to $\R_Y{\Lambda, n, y_0}$. The universal family parametrized by $\R_Y(\Lambda, n, y_0)$ induces naturally a universal family of stable framed marked $\Lambda$-modules parametrized by $\widetilde{\R}_{Y}(\Lambda, n, y_0)$, so it is a fine moduli space for the classification problem of these objects. 

We can extend naturally the $\GL(n,\CC)$-action on $\R_Y(\Lambda, n, y_0)$ to $\widetilde{\R}_Y(\Lambda, n, y_0)$. By construction, every point of $\widetilde{\R}_Y(\Lambda, n, y_0)$ is given by a stable marked $\Lambda$-module. Therefore, using Lemma \ref{lm stable marked implies simple}, it is possible to show that every point of $\widetilde{\R}_Y(\Lambda, n, y_0)$ has trivial $\GL(n,\CC)$-stabilizer. Recall from Theorem \ref{tm M=R/GL} that there exists a line bundle $L$ on $\R_Y(\Lambda, n, y_0)$, carrying a linearization $L$ of the $\GL(n, \CC)$-action, for which every point is semistable. We can define a linearization $\widetilde{L}$ on $\widetilde{\R}(\Lambda, n, y_0)$ by considering $r^*L^{\otimes b}$, for $b$ big enough. We claim that all the points of $\widetilde{\R}_Y(\Lambda, n, y_0)$ are semistable for the $\GL(n,\CC)$-action with respect to this linearization. To prove this claim, we first observe that the $\GL(n,\CC)$-stabilizer of any point $\widetilde{t} = (\Ff, \theta, \varphi, \sigma)$ of $\widetilde{\R}_Y(\Lambda, n, y_0)$ is trivial. This follows from the fact that, Lemma \ref{lm stable marked implies simple}, $(\Ff, \theta, \sigma)$ has trivial automorphism group, and an argument analogous to the one we used in the proof of Theorem \ref{tm GIT description of N} when we showed that the $\GL(V)$-stabilizer of a stable marked $\Lambda$-module is trivial. Suppose that $\widetilde{t}$ is not semistable, this implies that, for some $\widetilde{l} \in \widetilde{L}|_{\widetilde{t}}$, the closure of the orbit $\GL(n,\CC) \cdot (\widetilde{t}, \widetilde{l})$ intersects the zero section of $\widetilde{L}$. Since the $\GL(n,\CC)$-stabilizer of $\widetilde{t}$ is trivial, the orbit $\GL(n,\CC) \cdot (\widetilde{t}, \widetilde{l})$ is closed inside $\widetilde{\R}_Y(\Lambda, n, y_0)$, and then, there exists $(\widetilde{t}', 0)$ in $\GL(n,\CC)\cdot (\widetilde{t}, \widetilde{l})$. Let $t = r(\widetilde{t})$ and $l \in L^{\otimes b}|_{t}$ such that $\widetilde{l}$ projects to it, then we have that the $\GL(n,\CC)$-orbit of $(t, l)$ meets the zero section of $L^{\otimes b}$ and then, $t$ is not stable with respect to the linearization $L^{\otimes b}$, and neither with respect to $L$. This contradicts the statement of Theorem \ref{tm M=R/GL} that every point of $\R_Y(\Lambda, n, y_0)$ is semistable for this action, so $\widetilde{t}$ is semistable.

The universal family parametrized by $\widetilde{\R}(\Lambda, n, y_0)$ induces, by forgetting the fra\-ming, a family of marked $\Lambda$-modules with the local universal property. As a consequence, the identification \eqref{eq description of N as a GIT quotient} of the GIT quotient with the moduli space follows. With this identification, the morphism $r$ induces $\rho$ naturally.
\end{proof}

\subsection{Marked $\Lambda$-modules on abelian varieties}

Let $X$ denote an abelian variety and $x_0 \in X$ its identity element. A {\em marked $\Lambda$-module} on $X$ is a $\Lambda$-module marked on $x_0$. In this case we denote the moduli stack and moduli space simply by $\Nnn^{\st}_X(\Lambda, n)$ and $\N_X(\Lambda, n)$. The objective of this section is the explicit description of both $\Nnn^{\st}_X(\Lambda, n)$ and $\N_X(\Lambda, n)$. To do so, we recall the Hilbert stack of finite subsets of a quasi-projective scheme $Z$,
\[
\Hhh(Z, n) \, : \, \Aff^\op \longrightarrow \Groupoids,
\]
which associates to each affine scheme $S$ the category of closed subschemes $B \subset Z \times S$, where the projection $\pi: B \to S$, induced by $Z \times S \to S$, is flat and for all $s \in S$, the Hilbert polynomial is $P_{\Oo_{B_s}} = n$. Recall that, for any quasi-projective scheme $Z$, the Hilbert scheme $\Hilb^n(Z)$ is the scheme representing the functor 
\[
\Hhh(Z, n)/\! \cong \, : \, \Aff^\op \longrightarrow \Sets,
\]
obtained by taking the isomorphism classes of the groupoids $\Hhh(Z,n)(S)$.

We have now the ingredients to prove the second main result of the paper, which naturally implies Theorem \ref{tm N = Hilb}. 

\begin{theorem} \label{tm Nnn = Hhh}
There exists an isomorphism of stacks
\[
\Nnn^{\st}_{X}(\Lambda^\alpha, n) \cong \Hhh(\hat{X}^\alpha, n),
\]
given by the equivalence of categories $\Phi^{\alpha}$ from Corollary \ref{co Polishchuk-Rothstein transform}.
\end{theorem}

\begin{proof}
First of all, recall from Theorem \eqref{tm Mmm = Ttt} that $\Phi^{\alpha}$ provides an equivalence of categories between $\Mmm^{\sst}_X(\Lambda^\alpha, n)(S)$, the category of semistable $\Lambda^\alpha$-modules over $X_S = X \times S$, and $\Ttt(\hat{X}^\alpha,n)(S)$, the category of torsion $S$-sheaves of length $n$ on $\hat{X}^\alpha$. In particular, semistable $\Lambda^\alpha$-modules of trivial Chern classes are WIT.

An object of $\Hhh(\hat{X}^\alpha,n)(S)$ can be seen as pair $(\Gg, \Sigma)$, given by a torsion sheaf $\Gg$ on $\hat{X}^\alpha_S$ of length $n$, together with a subscheme structure, i.e. a surjective morphism,
\begin{equation} \label{eq subscheme structure}
\Sigma: \Oo_{\hat{X}^\alpha_S} \twoheadrightarrow \Gg.
\end{equation}

Let $(\Ff,\theta)$ be the $\Lambda^{\alpha}$-module over $X_S$ obtained from $\Gg$ under the transformation $\Phi^{\alpha}$. Since $\Gg$ is an extension of length $1$ sheaves, $(\Ff,\theta)$ is an extension of rank $1$ $\Lambda^{\alpha}$-modules of vanishing Chern classes, therefore $c_i(\Ff) = 0$.

We study now the transformation under $\Phi^{\alpha}$ of the subscheme structure \eqref{eq subscheme structure}. Recall the affine projection $\chi : \hat{X}^\alpha_S \to \hat{X}_S$ defined in Section \ref{sc rank 1}. One has that $\chi_*$ is an equivalence of categories, and \eqref{eq subscheme structure} is equivalent to
\[
\chi_* \Sigma: \chi_*\Oo_{\hat{X}^\alpha_S} \twoheadrightarrow \chi_*\Gg,
\]
where we recall that $\chi_*\Oo_{\hat{X}^\alpha_S}$ is isomorphic to $\hat{\Lambda}_{\alpha}$, which is an $\Oo_{X_S}$-algebra. Then, one has naturally $\Oo_{X_S} \stackrel{i}{\hookrightarrow} \hat{\Lambda}_{\alpha}$. Composing $i$ and $\chi_*\Sigma$ gives the morphism
\[
\Sigma_0 : \Oo_{\hat{X}_S} \longrightarrow \chi_*\Gg.
\]
Use the natural isomorphism $\hat{\Lambda}_{\alpha} \cong \hat{\Lambda}_{\alpha} \otimes \Oo_{\hat{X}_S}$, and note that $\chi_*\Sigma$ coincides with
\[
\hat{\Lambda}_\alpha = \hat{\Lambda}_\alpha \otimes \Oo_{\hat{X}_S} \stackrel{\Sigma_0}{\longrightarrow} \hat{\Lambda}_\alpha \otimes \chi_* \Gg \stackrel{\hat{\theta}}{\to} \chi_* \Gg\ ,
\]
where $\hat{\theta}$ is a $\hat{\Lambda}_\alpha$-module structure on $\chi_*\Gg$ satisfying $\chi_*\Sigma = \hat{\theta} \circ \Sigma_0$.

So the pair $(\Gg,\Sigma)$ is equivalent to the triple $(\chi_*\Gg,\hat{\theta}, \Sigma_0)$. Under this equivalence, it is clear that $\Sigma$ (equivalently, $\chi_*\Sigma$) is surjective if and only if there is no proper subsheaf $0 \neq \Gg' \subsetneq \chi_* \Gg$, preserved by $\hat{\theta}$ and such that
\[
\im \Sigma_0 \subseteq \Gg'.
\]
Note that this implies that $\Sigma_0$ is non-zero.

We have that $(\chi_*\Gg, \hat{\theta})$ gives $(\Ff, \theta)$ under $\Phi^{\alpha}$. By Corollary \ref{co Polishchuk-Rothstein transform}, the usual Fourier-Mukai functor $\Phi$, given in \eqref{eq definition of Phi}, and $\Phi^{\alpha}$ commute. Then, it is enough to transform $\Sigma_0$ under $\Phi$. Note that $\sigma' = \Phi^{-1}(\Sigma_0)$ is a morphism
\[
\sigma': \Oo_{\xi}[-d] \longrightarrow \Ff = \Phi^{-1}(\chi_* \Gg),
\]
where $\Oo_{\xi}$ is the sky-scraper sheaf over the $S$-point $\xi : S \to X_S$ and $\Ff$ is a locally free rank $n$ sheaf over $X_S$.

Serre duality implies that $\sigma'$ is equivalent to a morphism
\[
\sigma: (\Ff_{\xi})^\vee \longrightarrow \CC,
\]
which is the same as an element of $H^0(S,\xi^*\Ff)$. We have that $(\chi_*\Gg, \hat{\theta}, \Sigma_0)$ corres\-ponds to a marked $\Lambda^{\alpha}$-module $(\Ff,\theta, \sigma)$. 

Suppose that $(\Ff, \theta, \sigma)$ is not stable. Then there exists a proper subbundle $\Ff' \subset \Ff$ preserved by $\theta$ and such that $\sigma \in H^0(S,\xi^*\Ff')$. Then $\sigma$ is a morphism,
\[
\sigma: (\Ff'_{\xi})^\vee \longrightarrow \CC,
\]
and therefore the image of $\Sigma_0$ is contained in the proper subsheaf $\Phi(\Ff') \subsetneq \chi_*\Gg$, contradicting the fact that $\Sigma$ is a surjection.

Analogously, one can start with a stable marked $\Lambda^{\alpha}$-module $(\Ff,\theta,\sigma)$ on $X_S$. The image under $\Phi^{\alpha}$ of $(\Ff,\theta)$ gives $\Gg = (\chi_*\Gg,\hat{\theta})$. Since we require the existence of a global Jordan-H\"older filtration, Proposition \ref{pr stable implies rank 1} gives us that $(\Ff,\theta)$ is the extension of rank $1$ $\Lambda^{\alpha}$-modules, so $\Gg$ is the extension of sheaves of length $1$ and therefore, $G$ is a length $n$ sheaf. From the previous arguments, $\sigma$ gives a morphism
\[
\Sigma_0 : \Oo_{\hat{X}_S}  \to \chi_*\Gg,
\]
whose image is not contained in any proper subbundle $0 \neq \Gg' \subsetneq \chi_* \Gg$ preserved by $\hat{\theta}$, due to the stability of $(\Ff,\theta,\sigma)$. We have seen that this implies that $\Sigma_0$ defines a subscheme structure $\Sigma$ on $\Gg$. Then, the stable $(\Ff,\theta, \sigma)$ gives the triple $(\chi_*\Gg, \hat{\theta}, \Sigma_0)$, which determines uniquely an object of $\Hhh(\hat{X}^\alpha,n)(S)$.
\end{proof}

\begin{proof}[Proof of Theorem \ref{tm N = Hilb}.]
This is a consequence of Theorem \ref{tm Nnn = Hhh} and the fact that $\Hilb^n(\hat{X}^\alpha)$ represents the functor $\Hhh(\hat{X}^\alpha, n)/\! \cong$.
\end{proof}

Once we have a description of the marked moduli space, $\N_X(\Lambda^\alpha, n)$, we study its relation with the usual moduli space, $\M_X(\Lambda^\alpha, n)$. After Theorems \ref{tm M=R/GL} and \ref{tm marked R} and Remark \ref{rm M^LF = M}, we have the surjection 
\[
\rho : \N_X(\Lambda, n) \longrightarrow \M_X(\Lambda, n),
\]
which sends the isomorphism class of the stable marked $\Lambda^\alpha$-module $(\Ff, \theta, \sigma)$, to the $\Ss$-equivalence class of the semistable $\Lambda^\alpha$-module $(\Ff, \theta)$.

\begin{lemma} \label{eq commutative diagram of marked objects}
The following diagram is commutative,
\[
\xymatrix{
\Hilb^n(\hat{X}^\alpha) \ar[r]^{\nu}_{\cong} \ar[d]_{\delta} & \N_X(\Lambda^{\alpha}, n) \ar[d]^{\rho}
\\
\Sym^n(\hat{X}^\alpha) \ar[r]^{\mu}_{\cong} & \M_X(\Lambda^{\alpha}, n),
}
\]
being $\delta$ the Hilbert-Chow morphism.
\end{lemma}

\begin{proof}
The proof follows from the observation that both morphisms, $\mu$ and $\nu$, are constructed using the Fourier-Mukai transform $\Phi^{\alpha}$. 
\end{proof}

\section{Non-abelian Hodge theory and ADHM data}
\label{sc non-abelian Hodge theory}

Given a smooth projective scheme $Y$, we have that modules for the Dolbeault $D$-algebra $\Lambda^{\Dol} \cong \Sym^\bullet_{\Oo_Y}(\Tt Y)$ correspond to Higgs sheaves over $Y$. We abbreviate the {\it Dolbeault moduli space} $\M_Y(\Lambda^\Dol,n)$ by $\M_Y^\Dol(n)$. On the other hand, the De Rham $D$-algebra $\Lambda^\DR = \Dd_Y$ is the algebra of diffe\-ren\-tial operators and $\Oo_Y$-coherent $\Lambda^\DR$-modules are vector bundles with a flat connection over $Y$. We write $\M_Y^\DR(n)$ for the {\it De Rahm moduli space} $\M_Y(\Lambda^\DR,n)$. Let us recall the main result of Non-abelian Hodge theory, proved as a collective effort of Narasimhan and Seshadri \cite{narasimhan&seshadri}, Donaldson \cite{donaldson1, donaldson}, Corlette \cite{corlette}, Hitchin \cite{hitchin_self} and Simpson \cite{simpson1, simpson2}. 

\begin{theorem}[\cite{simpson2} Theorem 7.18]
There is a homeomorphism 
\[
\M_Y^\DR(n) \stackrel{homeo.}{\cong} \M_Y^\Dol(n).
\]
\end{theorem}

The third element of Non-abelian Hodge theory is the {\it Betti moduli space} $\M_Y^\Betti(n)$ classifying representations of the fundamental group $\pi_1(Y, y_0)$ of rank $n$. It is defined as the affine GIT quotient,
\[
\M_Y^\Betti(n, y_0) := \GIT{\Hom \left ( \pi_1(Y, y_0), \GL(n,\CC) \right )}{\GL(n,\CC),}
\]
of the space of all representations, $\R^\Betti_Y(n, y_0) := \Hom \left ( \pi_1(Y, y_0), \GL(n,\CC) \right )$ by the adjoint action of $\GL(n,\CC)$. This GIT quotient is defined in terms of a linearization on the representation spaces constructed from the character 
\begin{equation} \label{eq definition of eta}
\eta = (\det)^\ell : \GL(n,\CC) \to \CC^*, 
\end{equation}
where $\ell$ is a positive integer. Note that $\M_Y^\Betti(n, y_0)$ is independent of the choice of $y_0 \in Y$ so we simply write $\M_Y^\Betti(n)$. The Riemann-Hilbert Correspondence, proved by Deligne for algebraic connections, states that:

\begin{theorem}[\cite{simpson2} Proposition 7.8] \label{tm Riemann-Hilbert}
There exists a complex analytic isomorphism
\[
\M^\Betti_Y(n) \stackrel{cx. an.}{\cong} \M_Y^\DR(n).
\]
\end{theorem}

When $X$ is an abelian variety, 
\[
\M_X^\Dol(1)\cong \Tt^*\hat{X},
\]
where we recall that the cotangent bundle of an abelian variety is trivial. Denote by $X^\natural$ the universal group extension of $\hat{X}$ by $\mathfrak{g}^*$, which is associated to the canonical element $\id \in \mathfrak{g}^* \otimes \mathfrak{g} \cong H^1(\hat{X}, \mathfrak{g}^*)$. One has
\[
\M_X^\DR(1) \cong X^\natural.
\]
In this case, \eqref{eq Phi^alpha} corresponds to the Laumon-Rothstein transform \cite{laumon, rothstein}.

\begin{corollary}
The Dolbeault moduli space (of semistable Higgs bundles) over the abelian variety $X$ is
\[
\M^{\Dol}_X(n) \cong \Sym^n(\Tt^*\hat{X}).
\]
The De Rham moduli space (of vector bundles with flat connections) is 
\[
\M^{\DR}_X(n) \cong \Sym^n(X^\natural).
\]
\end{corollary}

Denote by $\N_X^\Dol(n)$, $\N_X^{\DR}(n)$ the moduli spaces of marked objects $\N_X(\Lambda^\Dol, n)$, $\N_X(\Lambda^\DR, n)$.

Theorem \ref{tm N = Hilb} now give us the follo\-wing description.

\begin{corollary} \label{co description of marked moduli spaces}
The Dolbeault moduli space (of marked semistable Higgs bundles) over the abelian variety $X$ is
\[
\N^{\Dol}_X(n) \cong \Hilb^n(\Tt^*\hat{X}).
\]
The De Rham moduli space (of marked vector bundles with flat connections) is 
\[
\N^{\DR}_X(n) \cong \Hilb^n(X^\natural).
\]
\end{corollary}

Recall from \cite{simpson3} that one can construct a family $\M_X^{\Hod}(n) \stackrel{p}{\to} \AA^1$, flat over $\AA^1$, such that the fibre $p^{-1}(0)$ is isomorphic to $\M_X^{\Dol}(n)$ and the fibres $p^{-1}(\tau)$ for $\tau \neq 0$ are isomorphic to $\M_X^{\DR}(n)$. The space $\M_X^{\Hod}(n)$ is constructed from Theorem \ref{tm existence of moduli space}, with $\Lambda^{\Hod} = \Uu(\TT^\Hod)$, where $\TT^\lambda$ is the Lie algebroid over $X \times \AA^1$ with underlying vector bundle $p_X^*\Tt X$ (where $p_X : X \times \AA^1 \to X$ is the natural projection) and associated to $t \otimes \id \in \Oo_{\AA^1} \otimes \mathfrak{g}^* \otimes \mathfrak{g}$, where $t : \AA^1 \to \Oo_{\AA^1}$ is the tautological section. Note that the restriction to every slice $\TT^\Hod |_{X \times \{ \tau \}}$ is a UTAI Lie algebroid whose universal enveloping algebra is $\Lambda^\tau$, defined at the end of Section \ref{sc D-algebras tau-connections}. Recall from Section \ref{sc Simpson construction} that $\Lambda^\tau$-modules correspond to $\tau$-connections, and, in particular, to Higgs bundles when $\tau = 0$ or flat connections when $\tau = 1$. The results of the previous sections hold also in this relative case: one defines $X^{\Hod}$ as the extension of group schemes relative to $\AA^1$
\[
0 \to \mathfrak{g}^* \times \AA^1 \to X^{\Hod} \to \hat{X} \times \AA^1 \to 0,
\]
induced by the morphism $\mathfrak{g}^* \times \AA^1 \to \mathfrak{g}^* \times \AA^1$, $(v,\tau) \mapsto (\tau\cdot v, \tau)$, and one shows that the Fourier-Mukai-transform interchanges $\Lambda^{\Hod}$-modules on $X$ with $\Oo_{X^\Hod}$-modules on $\hat{X}$; the moduli spaces $\M_X^{\Hod}(n) \to \AA^1$ are isomorphic to $\Sym^n_{\AA^1} (X^{\Hod}) \to \AA^1$, while the moduli spaces of marked objects $\N_X^{\Hod}(n) \to \AA^1$ are isomorphic to $\Hilb^n_{\AA^1} (X^{\Hod}) \to \AA^1$. We observe that $\Hilb^n_{\AA^1}(X^{\Hod})|_\tau$ can be identified with the moduli space of marked $\tau$-connections on $X$. In particular, one obtains:

\begin{lemma}
The moduli spaces $\N^\Dol_X(n)$ and $\N^\DR_X(n)$ are deformation equivalent.
\end{lemma}

\begin{proof}
It is clear from the construction of $X^\Hod$ that 
\[
X^\Hod|_{0} \cong \hat{X} \times \mathfrak{g}^* \cong \Tt^*\hat{X}.
\]
Also, after scaling, we have for any $\tau \neq 0$ that $X^\Hod|_{\tau}$ is isomorphic to $X^\Hod|_{\id}$, so
\[
X^\Hod|_{\tau} \cong X^\natural.
\]
The lemma follows from the existence of the flat family $\Hilb^n_{\AA^1}(X^{\Hod}) \to \AA^1$, since
\[
\Hilb^n_{\AA^1}(X^{\Hod})|_0 \cong \Hilb^n(X^\Hod|_0) \cong \Hilb^n(\Tt^*\hat{X}) \cong \N_X^{\Dol}(n),
\]
while for any $\tau \neq 0$,
\[
\Hilb^n_{\AA^1}(X^{\Hod})|_\tau \cong \Hilb^n(X^\Hod|_\tau) \cong \Hilb^n(X^\natural) \cong \N_X^{\DR}(n).
\]
\end{proof}

The fundamental group of an abelian variety $X$ of dimension $d$ is $\pi_1(X, x_0) \cong \ZZ^{\times 2d}$. As a consequence, the Betti moduli spaces are easy to describe. Note that
\[
\R^{\Betti}_X(n, x_0) \cong \{ (B_1, \dots, B_{2d}) \in \GL(n,\CC)^{\times 2d} \textnormal{ such that } B_i B_j = B_j B_i \};
\]
and recall that $\M^{\Betti}_X(n) = \R^{\Betti}_X(n, x_0) /\!\!/ \GL(n,\CC)$, where $\GL(n,\CC)$ acts by simultaneous conjugation. Since $n$ commuting matrices have common eigenvalues, one has a bijective morphism,
\begin{equation} \label{eq bijection Sym^n to M^Betti}
\Sym^n\left( (\CC^*)^{2d} \right) \longrightarrow \M^{\Betti}_X(n).
\end{equation}
Theorem \ref{tm description of M_X} describes $\M_X^{\Dol}(n)$ as a normal variety, and therefore $\M^{\Betti}_X(n)$ is normal as well due to Isosingularity Theorem \cite[Theorem 10.6]{simpson2}. Then, thanks to Zariski's Main Theorem, the bijection \eqref{eq bijection Sym^n to M^Betti} gives an isomorphism.

\begin{corollary}
\label{co M_B is Sym C*}
\[
\M^{\Betti}_X(n) \cong \Sym^n\left( (\CC^*)^{2d} \right).
\]
\end{corollary}

Define a {\em marked representation of the fundamental group} to be a pair $(\rho,v)$, where $\rho : \pi_1(X, x_0) \to \GL(n,\CC)$ is a representation and $v \in \CC^n$ is a vector; denote by $\N_B(A,n)$ the moduli space of marked representations of the fundamental group, i.e. the GIT quotient 
\begin{equation} \label{eq definition of N^Betti}
\N^{\Betti}_X(n) :=  \GIT{(\R^{\Betti}_X(n, x_0) \times \CC^n)}{\GL(n,\CC),}
\end{equation}
associated to the character $\eta$ from \eqref{eq definition of eta}.

Following \cite{nakajima, henni&jardim}, let us define the {\em variety of ADHM data} as the subset of $(\End(\CC^n)^{\times m}) \times \CC^n$ given by commuting endomorphisms 
\[
V(m,n) := \{ (B_1, \dots, B_{m}) \in \End(\CC^n)^{\times m} : B_i,B_j = B_j B_i \} \times \CC^n,
\]
and consider the GIT quotient associated to the character $\eta$,
\begin{equation} \label{eq GIT definition of N mn}
N(m,n) := \GIT{V(m,n)}{\GL(n,\CC).}
\end{equation}

One has the following:

\begin{theorem}[\cite{henni&jardim},  \cite{nakajima} for the case $m = 2$] \label{tm N nm cong Hilb CCm} The semistable points of the GIT quotient \eqref{eq GIT definition of N mn} are given by tuples $(B_1, \dots, B_m, v) \in V(m,n)$ such that there is no subspace $W \subset \CC^n$ such that $v \in W$ and $B_i(W) \subset W$, for all $i = \{ 1, \dots, m \}$. All these points have trivial stabilizer and are indeed stable.

Furthermore, one has the isomorphism of schemes
\label{tm N cong Hilb k^n}
\[
N(m,n) \cong \Hilb^n(\CC^m).
\]

\end{theorem} 

Heuristically, the isomorphism above is given by taking a $n$-tuple of commuting matrices to the corresponding eigenvalues.
Since $\R_X^{\Betti}(n, x_0) \subset V(2d,n)$ correspond to ADHM data whose eigenvalues are all non-zero, one can prove the following:

\begin{corollary}  \label{co N_B is Hilb C*}
There is an isomorphism
\[
\N^{\Betti}_X(n) \cong \Hilb^n\left( (\CC^*)^{2d} \right),
\]
and the diagram
\[
\xymatrix{
\Hilb^n\left((\CC^*)^{2d}\right) \ar[rr]_{\cong}^{\nu^\Betti} \ar[d]_{\delta^{\Betti}} & & \N^{\Betti}_X(n)  \ar[d]^{\rho^\Betti}
\\
\Sym^n\left( (\CC^*)^{2d} \right) \ar[rr]_{\cong}^{\mu^\Betti} & & \M^{\Betti}_X(n),
}
\]
commutes, where $\delta^{\Betti}$ is the Hilbert-Chow morphism and $\rho^\Betti$ the map given by forgetting the mar\-king.
\end{corollary}

We can obtain a result analogous to Theorem \ref{tm Riemann-Hilbert}, a Riemann-Hilbert correspondence for marked objects.

\begin{theorem} \label{tm NAHT for N_B and N_DR}
One has the complex analytic isomorphism
\[
\N_X^\DR(n) \stackrel{cx.an.}{\cong} \N^\Betti_X(n).
\]
\end{theorem}

\begin{proof}
After Theorem \ref{tm N nm cong Hilb CCm}, we identify the semistable locus (stable indeed) of the GIT quotient \eqref{eq definition of N^Betti} with the open subset given by those marked representation of the fundamental group $(\rho, v)$ such that there is no $W \subset \CC^n$ such that $\rho(\pi_1(X, x_0))$ preserves $W$ and $v \in W$. In particular, this implies that $v$ is non-zero. Then, \eqref{eq definition of N^Betti} factors through 
\[
\N^\Betti_X(n) \cong \GIT{\left ( \R^\Betti_X(n, x_0) \times (\CC^n - \{ 0 \}) \right )^{st}}{\GL(n,\CC).}
\]

Recall from Theorem \ref{tm marked R} that $\N_X^\DR(n)$ is isomorphic to the $\GL(n, \CC)$-quotient of $\widetilde{\R}_X^\DR(n, x_0)$, which is defined as the open subset $\left ( \R^\DR_X(n, x_0) \times (\CC^n - \{ 0 \})\right )^{st}$, given by stable marked flat connections. From \cite[Theorem 7.1]{simpson2} there exists a complex analytic isomorphism between $\R_X^\DR(n, x_0)$ and $\R_X^\Betti(n, x_0)$ compatible with the action of $\GL(n,\CC)$. This complex analytic isomorphism extends trivially to $\R_X(n, x_0) \times \CC^n$ and one we can easily observe that it is compatible with the notion of stability for marked objects. This restricts to a complex analytic isomorphism between $\left ( \R^\DR_X(n, x_0) \times (\CC^n - \{ 0 \})\right )^{st}$ and $\left ( \R^\Betti_X(n, x_0) \times (\CC^n - \{ 0 \})\right )^{st}$, compatible with the $\GL(n,\CC)$-action, that implies the theorem. 
\end{proof}

\begin{remark}
Alternatively, using Corollaries \ref{co description of marked moduli spaces} and \ref{co N_B is Hilb C*}, one can prove Theorem \ref{tm NAHT for N_B and N_DR} noting that 
\[
\Hilb^n(X^\natural) \stackrel{cx.an.}{\cong} \Hilb^n((\CC^*)^{2d}),
\]
since both coincide with Douady space of $(X^\natural) \stackrel{cx.an.}{\cong} (\CC^*)^{2d}$.
\end{remark}

Summarizing the results of this section, we describe the non-abelian Hodge picture for marked objects
\begin{equation} \label{eq non-abelian Hodge theory for marked objects}
\xymatrix{
\Hilb^n((\CC^*)^{2d}) \ar@{<->}[r]^{\stackrel{cx.an.}{\cong}} \ar[d]^{\cong}_{\nu^{\Betti}} & \Hilb^n(X^\natural) \ar@{<..>}[r]^{de\!f.eq.} \ar[d]^{\cong}_{\nu^{\DR}}  & \Hilb^n(\Tt^*\hat{X}) \ar[d]^{\cong}_{\nu^{\Dol}} 
\\
\N_X^{\Betti}(n) \ar@{<->}[r]^{\stackrel{cx.an.}{\cong}} \ar@{->>}[d]_{\rho^\Betti} &  \N_X^\DR(n) \ar@{<..>}[r]^{de\!f.eq.} \ar@{->>}[d]_{\rho^\DR} & \N_X^\Dol(n) \ar@{->>}[d]_{\rho^\Dol}
\\
\M_X^{\Betti}(n) \ar[d]^{\cong}_{\xi^\Betti} \ar@{<->}[r]^{\stackrel{cx.an.}{\cong}} & \M_X^\DR(n) \ar[d]^{\cong}_{\xi^\DR} \ar@{<->}[r]^{\stackrel{homeo.}{\cong}}  & \M_X^\Dol(n) \ar[d]^{\cong}_{\xi^\Dol} 
\\
\Sym^n((\CC^*)^{2d}) \ar@{<->}[r]^{\stackrel{cx.an.}{\cong}}  & \Sym^n(X^\natural) \ar@{<->}[r]^{\stackrel{homeo.}{\cong}}  & \Sym^n(\Tt^*\hat{X}), 
}
\end{equation}
where the $\rho^*$ are the morphisms that forget the marking, $\xi^* = (\mu^*)^{-1}$ and the compositions $\xi^* \circ \rho^* \circ \nu^*$ are the corresponding Hilbert-Chow morphisms.   

\

\begin{remark}
When $X$ is an elliptic curve, the Hilbert schemes $\Hilb^n(\CC^* \times \CC^*)$, $\Hilb^n(X^\natural)$ and $\Hilb^n(\Tt^*\hat{X})$ are smooth. In that case, $\N^*_X(n) \stackrel{\rho^*}{\longrightarrow} \M^*_X(n)$ are the resolution of singularities of the usual moduli spaces.
\end{remark}

\end{document}